\documentclass[smallextended,numbook]{svjour3_mod}

\usepackage{amsmath}
\usepackage{amssymb}

\usepackage{mathptmx}
\usepackage{helvet}
\usepackage{courier}

\journalname{Numerische Mathematik}

\begin{document}

\title{On the approximation of electronic wavefunctions 
 by anisotropic Gauss and Gauss-Hermite functions}

\author{Stephan Scholz \and Harry Yserentant}

\institute{
 Institut f\"ur Mathematik, Technische Universit\"at Berlin,
 10623 Berlin, Germany\\
\email{stscholz@math.tu-berlin.de, yserentant@math.tu-berlin.de}}

\date{November 27, 2016}

\titlerunning{On the approximation of electronic wavefunctions}
 
\authorrunning{S. Scholz, H. Yserentant}

\maketitle


\begin{abstract}

The electronic Schr\"odinger equation describes the motion of $N$ 
electrons under Coulomb interaction forces in a field of clamped 
nuclei. The solutions of this equation, the electronic wavefunctions,
depend on $3N$ variables, three spatial dimensions for each electron. 
We study the approximability of these wavefunctions by linear 
combinations of anisotropic Gauss functions, or more precisely
Gauss-Hermite functions, products of polynomials and anisotropic 
Gauss functions in the narrow sense. We show that the original, 
singular wavefunctions can up to given accuracy and a negligibly 
small residual error be approximated with only insignificantly 
more such terms than their convolution with a Gaussian kernel of 
sufficiently small width and that basically arbitrary orders of 
convergence can be reached.

%
%
%
%
%

              
%
%
%
%
%
%

\end{abstract}


\renewcommand   {\thesection}{\arabic{section}}
\renewcommand   {\theequation}{\arabic{section}.\arabic{equation}}

\renewcommand   {\thelemma}{\arabic{section}.\arabic{lemma}}
\renewcommand   {\thetheorem}{\arabic{section}.\arabic{theorem}}

\spnewtheorem*  {unnumberedtheorem}{Theorem}{\bf }{\it}
\spnewtheorem*  {unnumberedlemma}{Lemma}{\bf }{\it}


\newcommand     {\rmref}[1]   {{\rm (\ref{#1})}}
\newcommand     {\rmthmref}[1]{{\rm  \ref{#1}}}
  
\newcommand     {\fourier}[1] {\widehat{#1}}
		 
\newcommand     {\vektor}[1]  {#1}
\newcommand     {\diff}[1]    {\mathrm{d}#1}

\def \x         {\vektor{x}}
\def \a         {\vektor{a}}

\def \k         {\vektor{k}}

\def \Q         {\vektor{Q}}
\def \I         {\vektor{I}}

\def \valpha    {\vektor{\alpha}}
\def \veta      {\vektor{\eta}}
\def \vomega    {\vektor{\omega}}

\def \D         {\mathcal{D}}
\def \S         {\mathcal{S}}

\def \e         {\mathrm{e}}
\def \i         {\mathrm{i}}

\def \dx        {\,\diff{\x}}
\def \domega    {\,\diff{\vomega}}
\def \dt        {\,\diff{t}}

\def \GS        {\widetilde{G}}
\def \VS        {\widetilde{V}}
\def \TS        {\widetilde{T}}

\def \uS        {\widetilde{u}}

\def \sumi      {\sum_{i=1}^N}
\def \sumij     {\sum_{\underset{\scriptstyle i\neq j}{i,j=1}}^N}

\def \Hs        {\mathcal{H}}
\def \E         {\mathcal{E}}

\def \wS        {\widetilde{w}}


\section{Introduction}
\label{sec1}

\setcounter{equation}{0}
\setcounter{lemma}{0}
\setcounter{theorem}{0}

The approximation of high-dimensional functions, whether they 
be given explicitly or implicitly as solutions of differential 
equations, represents one of the grand challenges of applied 
mathematics. This field made great progress during the past 
years, above all by the emergence of modern tensor product 
methods \cite{Hackbusch}. The astonishing efficiency of such
methods for the numerical solution of certain partial
differential equations and their obviously often rapid 
convergence can meanwhile be explained theoretically 
\cite{Dahmen-DeVore-Grasedyck-Sueli}. Adaptive techniques
\cite{Bachmayr-Dahmen} have been developed that enable to 
exploit this convergence behavior in practical computations.
One of the most notorious and complicated problems of this 
type, however, the electronic Schr\"odinger equation, largely
resists such approaches. The Schr\"odinger equation forms the 
basis of quantum mechanics and is of fundamental importance 
for our understanding of atoms and molecules. It links 
chemistry to physics and describes a system of electrons and 
nuclei that interact by Coulomb attraction and repulsion forces. 
As proposed by Born and Oppenheimer in the nascency of quantum 
mechanics, the slower motion of the nuclei is mostly separated 
from that of the electrons. This results in the electronic 
Schr\"odinger equation, the problem to find the eigenvalues 
and eigenfunctions of the electronic Hamilton 
operator\footnote{To simplify the presentation, we omit 
the usual factor $1/2$ in front of the kinetic energy 
part, which corresponds only to a minor change of the 
lengthscale and does not affect the mathematics.}
\begin{equation}    \label{eq1.1}
H\;=\,{}-\,\sum_{i=1}^N\,\Delta_i \;-\;
\sumi\sum_{\nu=1}^K\frac{Z_\nu}{|\x_i-\a_\nu|}
\,+\,\frac12 \sumij\frac{1}{|\x_i-\x_j|}.
\vspace{-0.5\belowdisplayskip}
\end{equation}
It acts on functions with arguments $\x_1,\ldots,\x_N$ in 
$\mathbb{R}^3$, which are associated with the positions 
of the considered electrons. The $\a_1,\ldots,\a_K$ in 
$\mathbb{R}^3$ are the fixed positions of the nuclei and 
the values $Z_\nu>0$ the charges of the nuclei in multiples 
of the electron charge. The reason for the comparatively 
low performance of tensor product methods when applied to 
the electronic Schr\"odinger equation is that such methods 
fix a set of directions. It is not possible with tensor 
product methods to capture simultaneously and equally well 
the singularities arising from the interaction of the 
electrons and the nuclei aligned with the coordinate 
directions and from the electron-electron singularities 
aligned with the diagonals.

Therefore we come back to an old idea, the approximation 
of the eigenfunctions of the Schr\"odinger operator 
(\ref{eq1.1}), the electronic wavefunctions, by Gauss 
functions. The almost exclusive use of Gauss functions 
in quantum chemistry \cite{Helgaker} is partly motivated 
by the fact that the arising integrals can be easily 
evaluated but has also to do with their good approximation 
properties. We study the approximability of electronic 
wavefunctions by linear combinations of Gauss-Hermite 
functions
\begin{equation}    \label{eq1.2}    
P(\x)\exp\left(-\frac12\,(\x-\a)\cdot\Q(\x-\a)\right)
\end{equation}
composed of a polynomial part $P(\x)$ and an anisotropic
Gauss function. We will denote such functions in what 
follows shortly as Gauss functions. The symmetric positive 
definite matrices $Q$ are arbitrary and not fixed in 
advance. The same holds for the points $\a\in\mathbb{R}^{3N}$ 
around which the Gauss functions are centered and which 
are only indirectly determined by the positions of the 
nuclei. The ansatz in particular covers products of 
Gaussian orbitals and Gaussian geminals.
The key point is that the set of these functions 
is invariant to linear transformations and shifts of the 
coordinate system. There is therefore no need to distinguish 
between the singularities arising from the interaction of the 
electrons and the nuclei and among the electrons themselves. 
Basically we show that electronic wavefunctions can be 
approximated with arbitrary order in the number of the 
involved terms by linear combinations of such Gauss functions. 
This behavior is in contrast to approximation results that 
are based on the mixed regularity of the wavefunctions; 
see \cite{Ys_2004,Ys_2010,Ys_2011} and 
\cite{Kreusler-Yserentant}. These regularity properties 
together with the antisymmetry of the wavefunctions enforced 
by the Pauli principle suffice to construct approximations 
by linear combinations of Slater determinants composed of 
a fixed set of basis functions that converge with an order 
in the number of the involved terms that does not deteriorate 
with the number of electrons. The attainable convergence 
order is, however, fundamentally limited by the limited 
regularity of the wavefunctions. In the present work, we 
do not make direct use of regularity properties of the 
wavefunctions. The approximations are constructed 
directly by a kind of iterative procedure.

The key to our approximation of the wavefunctions is 
the extremely accurate approximation of the functions 
$1/\sqrt{r}$ and $1/r$ by exponential functions and with 
that indirectly also that of $1/r$ by Gauss functions. 
Approximations of this kind form a rather universal tool 
that received much attention during the past years. 
Braess \cite{Braess} and Kutzelnigg \cite{Kutzelnigg} 
studied such approximations in view of applications in 
quantum mechanics, particularly regarding the hydrogen 
ground state. Bachmayr, Chen, and Schneider 
\cite{Bachmayr-Chen-Schneider} extended this work to 
excited states. Braess and Hackbusch 
\cite{Braess-Hackbusch} have shown that the best uniform
approximation of $1/r$ on intervals $[R,\infty)$, $R>0$, 
by finite linear combinations of exponential functions
converges almost exponentially in the number of the  
involved terms. A detailed exposition of such 
approximations and results and a survey on some of their 
applications can be found in \cite{Braess-Hackbusch_2}. 

The central idea is to approximate the Coulomb potentials 
in the operator (\ref{eq1.1}) and the inverse of the 
correspondingly shifted Laplace operator, expressed in 
terms of the Fourier transform, with very high relative 
accuracy by series of Gauss functions. The original 
eigenvalue problem is at first rewritten as a linear 
equation with the convolution $f=K*u$ of the eigenfunction 
$u$ under consideration with a Gaussian kernel $K$ of 
sufficiently small width as right hand side. The solution 
of the corresponding approximate equation with same right 
hand side is then orders of magnitude closer to the 
solution $u$ of the original equation than $u$ to the 
true physical wavefunction influenced, among other things, 
by relativistic effects and spin-orbit coupling. The 
approximate equation is solved via a Neumann series. 
This series is, after expansion of the right hand side 
into a series of Gauss functions, itself a series of 
Gauss functions that is truncated in an appropriate manner. 
The main hurdle, to which most of this paper is devoted, 
is the control of this truncation process.

The main result can be roughly sketched as follows. 
Assume that there exists an infinite sequence 
$g_1,g_2,\ldots$ of Gauss functions such that 
for every $\varepsilon>0$
\begin{equation}    \label{eq1.3}
\Big\|\,K*u\,-\sum_{j=1}^n g_j\,\Big\|_1\leq\;\varepsilon, 
\quad
n\;\leq\,\Big(\frac{\kappa}{\varepsilon}\Big)^{1/r},
\end{equation}
where $r$ is a given approximation order and the norm is
the $H^1$-norm, the norm associated with the Schr\"odinger 
equation. Such expansions can be constructed using the 
exponential decay of the wavefunctions, and when indicated 
also of their integer and fractional order mixed derivatives
\cite{Kreusler-Yserentant,Ys_2010,Ys_2011}. The solution 
of the approximate equation, which serves as a quasi-exact 
substitute of the original wavefunction $u$, can then, for 
arbitrarily small $\varepsilon>0$, be approximated by a 
linear combination of 
\begin{equation}    \label{eq1.4}
n\;\leq\,2\;\Big(\frac{2\kappa}{\varepsilon}\Big)^{1/r}
\end{equation}
Gauss functions of the same polynomial degree up to an 
$H^1$-error $\varepsilon$, provided the width of the smoothing 
kernel $K$ is sufficiently small in dependence of the 
approximation order $r$. The approximation of the original, 
singular wavefunction $u$ up to a very small, negligible residual 
error determined by the approximations of the Coulomb potentials 
and the inverse of the shifted Laplace operator thus does not 
require substantially more terms than that of its smoothed 
variant $K*u$. 

The rest of this paper is organized as follows. In
Sect.~\ref{sec2}, the Schr\"odinger equation is precisely 
stated, as an eigenvalue problem in weak form on the 
Sobolev space~$H^1$. Some properties of its solutions 
are shortly discussed. More information on the mathematics 
behind the electronic Schr\"odinger equation can be found 
in \cite{Ys_2010}. In Sect.~\ref{sec3}, the equation is 
brought into the form from which our approximation result 
is derived. Sect.~\ref{sec4} discusses the sensitivity of the 
transformed equation from Sect.~\ref{sec3} to perturbations 
of the interaction potential and the inverse of the shifted 
Laplace operator. In Sect.~\ref{sec5}, we analyze 
approximations due to Beylkin and Monz\'{o}n 
\cite{Beylkin-Monzon} of the functions $1/r^{\,\beta}$ by 
series of exponential functions. The essential point here 
is that we have explicit access to the expansion coefficients 
and that these are intimately connected with a scale of 
Sobolev norms. These approximations are used in 
Sect.~\ref{sec6} to set up the mentioned approximations of 
the Coulomb potentials and the inverse of the shifted 
Laplace operator. Sect.~\ref{sec7} is of rather technical 
nature and devoted to estimates of the norms of the parts 
into which the approximations of the single operators 
split. The key point is that these norms decay exponentially,
a property that forms the basis of our final approximation 
result in Sect.~\ref{sec8}.


\section{The weak form of the equation}
\label{sec2}

\setcounter{equation}{0}
\setcounter{lemma}{0}
\setcounter{theorem}{0}

The solution space of the electronic Schr\"odinger equation 
is the Hilbert space $H^1$ that consists of the one times
weakly differentiable, square integrable functions 
\begin{equation}    \label{eq2.1}
u:(\mathbb{R}^3)^N\!\to\,\mathbb{R}:
(\x_1,\ldots,\x_N)\to u(\x_1,\ldots,\x_N)
\end{equation}
with square integrable first-order weak derivatives. The
norm $\|\cdot\|_1$ on $H^1$ is composed of the $L_2$-norm
$\|\cdot\|_0$ and the $H^1$-seminorm $|\cdot|_1$, the
$L_2$-norm of the gradient. The space $H^1$ is the space 
of the wavefunctions for which the total position 
probability remains finite and the expectation value of 
the kinetic energy can be given a meaning. By $\D$ we 
denote the space of all infinitely differentiable functions
(\ref{eq2.1}) with bounded support. The functions in $\D$ 
form a dense subset of $L_2$ and of $H^1$ as well. Before 
we can state the equation, we have to study the potential
\begin{equation}    \label{eq2.2}
V(\x)\;=\;-\;\sumi\sum_{\nu=1}^K\frac{Z_\nu}{|\x_i-\a_\nu|}
\,+\,\frac12 \sumij\frac{1}{|\x_i-\x_j|}
\vspace{-0.5\belowdisplayskip}
\end{equation}
in the Schr\"odinger operator (\ref{eq1.1}) that is composed
of the nucleus-electron interaction potential, the first
term in (\ref{eq2.2}), and the electron-electron interaction 
potential. 
\begin{lemma}       \label{lm2.1}
For arbitrary functions \rmref{eq2.1} in $\D$ and with 
that also $H^1$,
\begin{equation}    \label{eq2.3}
\|Vu\|_0\,\leq\,(2\,Z+N-1)\,N^{\,1/2}\,|\,u\,|_1,
\end{equation}
where $Z=\sum_\nu Z_\nu$ is the total charge of the
nuclei.
\end{lemma}
The proof of Lemma~\ref{lm2.1} is based on the 
three-dimensional Hardy inequality
\begin{equation}    \label{eq2.4}
\int\!\frac{1}{\,|\x|^2}\,v^2\dx \,\leq\,
4\int\!|\nabla v|^2\dx
\end{equation}
for infinitely differentiable functions 
$v:\mathbb{R}^3\to\mathbb{R}$ with compact support.
By (\ref{eq2.3}),
\begin{equation}    \label{eq2.5}
a(u,v)=\int\big\{\nabla u\cdot\nabla v+Vuv\big\}\dx\,=\,(Hu,v)
\end{equation}
is a $H^1$-bounded bilinear form on $\D$, where 
$(\cdot\,,\cdot)$ is the $L_2$-inner product. It can 
be uniquely extended to a bounded bilinear form on 
$H^1$. In this setting, a function $u\neq 0$ in $H^1$ 
is an eigenfunction of the Schr\"odinger operator 
(\ref{eq1.1}) for the eigenvalue $\lambda$ if 
\begin{equation}    \label{eq2.6}
a(u,v)=\lambda (u,v), \quad v\in H^1.
\end{equation}
The weak form (\ref{eq2.6}) of the eigenvalue equation 
$Hu=\lambda u$ in particular fixes the behavior of the 
eigenfunctions at the singularities of the interaction 
potential and at infinity. For normed $u$, $a(u,u)$ 
is the expectation value of the total energy.  

It should be noted that only those eigenfunctions are 
physically admissible that are antisymmetric with 
respect to the permutation of the positions $\x_i$ of 
electrons of the same spin. This is a consequence of 
the Pauli principle. We will not utilize this property. 
We will, however, restrict ourselves to bound states 
of the system under consideration, eigenfunctions $u$ 
for eigenvalues $\lambda$ below the ionization threshold, 
a quantity that is bounded from above by the value zero. 
Such eigenfunctions and their first order weak derivatives 
decay exponentially in the $L_2$-sense. There is a 
constant $\mu>0$, depending on the distance of the 
given eigenvalue $\lambda$ to the ionization threshold, 
such that the function
\begin{equation}    \label{eq2.7}
\x\;\to\;\e^{\,\mu|\,\x\,|}u(\x)
\end{equation}
is square integrable and even possesses square integrable
first and certain higher integer and fractional order 
weak derivatives. This means among other things that the 
Fourier transforms of these eigenfunctions are 
real-analytic and that their partial derivatives of 
arbitrary order are bounded. Such properties and the fact 
that the eigenvalues under consideration are less than 
zero will play an essential role in our reasoning. 
More details, about this and on the electronic 
Schr\"odinger equation in general, can be found in 
\cite{Ys_2010} and, concerning additional information 
on the regularity properties of the exponentially 
weighted wavefunctions, in \cite{Kreusler-Yserentant} 
and \cite{Ys_2011}.


\section{An operator version}
\label{sec3}

\setcounter{equation}{0}
\setcounter{lemma}{0}
\setcounter{theorem}{0}

We will fix the eigenvalue $\lambda<0$ under consideration 
for the rest of this paper and introduce at first the 
inverse of the correspondingly shifted Laplace operator 
\mbox{$-\Delta-\lambda$}, a bounded linear operator $G$ 
from the space $L_2$ of the square integrable functions to 
the Sobolev space $H^2$ of the square integrable functions 
with square integrable first and second order weak 
derivatives. For rapidly decreasing functions $f$,
\begin{equation}    \label{eq3.1}
(Gf)(\x)\,=\,\Big(\frac{1}{\sqrt{2\pi}}\Big)^{3N}\!
\int \frac{1}{\,|\vomega|^2-\,\lambda}\;
\fourier{f}\,(\vomega)\,\e^{\,\i\,\vomega\,\cdot\,\x}\domega.
\end{equation}
For $u\in H^1$ and infinitely differentiable functions 
$v\in\D$, with compact support,
\begin{equation}    \label{eq3.2}
(u,G^{-1}v)\;=\,\int\nabla u\cdot\nabla v\dx\;-\,\lambda (u,v).
\end{equation}
For all square integrable functions $f$ and $g$,
\begin{equation}    \label{eq3.3}
(Gf,g)\,=\,(f,Gg).
\end{equation}
\begin{lemma}       \label{lm3.1}
A function $u\in H^1$ solves the eigenvalue equation 
\rmref{eq2.6} if and only if
\begin{equation}    \label{eq3.4}
u\;+\,G\,Vu\;=\;0
\end{equation}
and is therefore automatically contained in $H^2$.
\end{lemma}
\begin{proof}
As follows from (\ref{eq3.2}) and (\ref{eq3.3}), for 
all $u\in H^1$ and all $v\in\D$
\begin{displaymath}
(u\;+\,G\,Vu,G^{-1}v)\,=\,a(u,v)-\lambda(u,v).
\end{displaymath}
Since $\D$ is dense in $H^1$, a function $u\in H^1$
that satisfies the equation (\ref{eq3.4}) solves 
therefore also the eigenvalue equation (\ref{eq2.6}).
If the function $u\in H^1$ solves conversely the 
equation (\ref{eq2.6}), for all functions $v\in\D$
\begin{displaymath}
0\;=\;a(u,Gv)-\lambda(u,Gv)\,=\,(u\;+\,G\,Vu,G^{-1}Gv)
\,=\,(u\;+\,G\,Vu,v).
\end{displaymath}
As $\D$ is dense in $L_2$, $u$ solves therefore the
equation (\ref{eq3.4}). Since the multiplication 
operator $V$ maps $H^1$ into $L_2$ and the operator
$G$ the space $L_2$ into $H^2$, this means at the 
same time that $u$ is contained in $H^2$.
\qed
\end{proof}

In the next step, we split the potential part $Vu$ into 
the sum of a smooth part $QVu$, the convolution $K*Vu$
of $Vu$ with a Gaussian kernel
\begin{equation}    \label{eq3.5}
K(\x)\,=\,\left(\frac{1}{4\pi\gamma}\right)^{3N/2}\!
\exp\left(-\frac{1}{4\gamma}\;|\,\x\,|^2\right)
\end{equation}
of a width that is determined by the constant $\gamma<1$ 
and will later be adapted to the needs, and the 
complementary part $PVu$. The equation (\ref{eq3.4}), 
which is by Lemma~\ref{lm3.1} equivalent to the original 
Schr\"odinger equation (\ref{eq2.6}), turns then into
\begin{equation}    \label{eq3.6}
u\;+\,GPVu\;=\,-\,GQVu.
\end{equation}
The operators $G$, $Q$, and $P=I-Q$ commute. This 
follows from the representation
\begin{equation}    \label{eq3.7}
(Qf)(\x)\,=\,\Big(\frac{1}{\sqrt{2\pi}}\Big)^{3N}\!
\int \e^{-\gamma\,|\vomega|^2}\,
\fourier{f}\,(\vomega)\,\e^{\,\i\,\vomega\,\cdot\,\x}\domega
\end{equation}
of the convolution operator in terms of the Fourier 
transform. As the wavefunction~$u$ under consideration 
satisfies the equation (\ref{eq3.4}), therefore 
$GQVu=-\,Qu$. Thus 
\begin{equation}    \label{eq3.8}
u\;+\,GPVu\;=\,Qu.
\end{equation}
The point is that, for sufficiently small $\gamma$, 
that is, a sufficiently small width of the smoothing 
kernel, the solution $u$ is completely determined 
by its regular part $Qu$.
\begin{lemma}       \label{lm3.2}
For all square integrable functions $f$ and all $\gamma<1$,
\begin{equation}    \label{eq3.9}
\|GPf\|_1\,\leq\,\sqrt{\gamma}\;\|f\|_0.
\end{equation}
\end{lemma}
\begin{proof}
The proof is a rather immediate consequence of the Fourier 
representation
\begin{displaymath}
\|GPf\|_1^2\;=\,\int\big(1+|\vomega|^2\big)\bigg(
\frac{1-\,\e^{-\gamma\,|\vomega|^2}}{|\vomega|^2-\lambda}
\bigg)^{\!2}\,|\fourier{f}(\vomega)|^2\domega
\end{displaymath}
of the square of the norm to be estimated, of the 
elementary estimate
\begin{displaymath}
(1+\,t)\left(\frac{1-\,\e^{-t}}{t}\right)^2\leq\,1
\end{displaymath}
that holds for all $t>0$, and of Plancherel's theorem.
\qed
\end{proof}
The $L_2$-norm of $Vu$ can by (\ref{eq2.3}) be estimated 
by the $H^1$-norm of $u$. The operator
\begin{equation}    \label{eq3.10}
T:H^1\to \,H^1:u\,\to\,GPVu
\end{equation}
is thus for sufficiently small $\gamma$ contractive, 
more precisely, introducing the constant
\begin{equation}    \label{eq3.11}
\theta(N,Z)\,=\,(2\,Z+N-1)\,N^{\,1/2}
\end{equation}
already known from Lemma~\ref{lm2.1}, if the condition
\begin{equation}    \label{eq3.12}
\sqrt{\gamma}\;\theta(N,Z)\,<\,1
\end{equation}
is satisfied. For any given $f\in H^1$, and in 
particular for $f=Qu$,  the equation
\begin{equation}    \label{eq3.13}
u\,+Tu\,=\,f
\end{equation}
possesses then the uniquely determined solution
\begin{equation}    \label{eq3.14}
u\;=\;(I+T)^{-1}f\,=\,\sum_{\nu=0}^\infty (-1)^\nu T^\nu\!f.
\end{equation}
This representation of the wavefunction $u$ represents
the foundation of our theory.


\section{The influence of perturbations}
\label{sec4}

\setcounter{equation}{0}
\setcounter{lemma}{0}
\setcounter{theorem}{0}

The basic idea is the approximation of the symbols of 
the operators of which the operator $T$ is composed 
with extremely high accuracy by series of Gauss 
functions. In this section we study how the solution 
of the equation (\ref{eq3.13}) reacts to such 
perturbations of the operator $T$. Our starting point 
will be approximations
\begin{equation}    \label{eq4.1}
\VS:H^1\to\,L_2, \quad \GS:L_2\,\to\,H^2
\end{equation}
of the multiplication operator $u\to Vu$ and of the 
operator (\ref{eq3.1}) for which
\begin{equation}    \label{eq4.2}
\|Vu-\VS u\|_0\,\leq\,\theta(N,Z)\,\varepsilon\,|\,u\,|_1,
\quad
\|Gf-\GS f\|_1\,\leq\,\varepsilon\,\|Gf\|_1
\end{equation}
holds for all $u\in H^1$ and $f\in L_2$, respectively, 
where $\varepsilon$ is a very small constant, by orders 
of magnitude less than the accuracy of the physical 
model and independent of the given eigenvalue. The 
difference of the perturbed operator 
\begin{equation}    \label{eq4.3}
\TS\,=\,\GS P\VS
\end{equation}
and the operator (\ref{eq3.10}) attains then the 
representation
\begin{equation}    \label{eq4.4} 
\TS\,-\,T\;=\;(\GS-G)PV\,+\,GP(\VS-V)\,+\,(\GS-G)P(\VS-V),
\end{equation}
from which by assumption (\ref{eq4.2}) at first 
the estimate
\begin{equation}    \label{eq4.5}
\|\TS u\,-Tu\|_1\;\leq\;
\varepsilon\,\|GPVu\|_1\,+\,(1+\varepsilon)\|GP(\VS u-Vu)\|_1
\end{equation}
follows. It implies by (\ref{eq3.9}), Lemma~\ref{lm2.1}, 
and (\ref{eq4.2}) the estimate
\begin{equation}    \label{eq4.6}
\|\TS u\,-Tu\|_1\,\leq\,\delta\,|\,u\,|_1, \quad
\delta\,=\,
\theta(N,Z)\,\gamma^{\,1/2}\,(2\varepsilon+\varepsilon^2).
\end{equation}
That is, the perturbed operator (\ref{eq4.3}) is already 
for $\delta<1-\|T\|_1$ contractive and differs with the
given $\delta$ only very little from $T$. The distance 
between our eigenfunction~$u$, the solution of the 
original equation (\ref{eq3.8}) here rewritten as
\begin{equation}    \label{eq4.7}
u\,+\TS u\,=\,Qu\,+(\TS u-Tu),
\end{equation}
and the solution $\uS$ of the modified equation
\begin{equation}    \label{eq4.8}
\uS\,+\TS\uS\,=\,Qu,
\end{equation}
in which the term $\TS u-Tu=P(\GS\VS u-GVu)$ is neglected, 
satisfies then the estimate
\begin{equation}    \label{eq4.9}
\|u\,-\uS\,\|_1\,\leq\,\frac{\delta}{1-\|\TS\|_1}\;|\,u\,|_1.
\end{equation}
The accuracy of the approximation of the operators thus 
transfers almost completely to the solution $\uS$ of 
the perturbed equation.

The smooth part $Qu$ of the eigenfunction $u$, its convolution 
with the kernel (\ref{eq3.5}), reflects the global structure 
of $u$. The level of resolution is determined by the width of 
the kernel. The transition to the solution $\uS$ of the perturbed
equation (\ref{eq4.8}) adds the missing information on the 
singularities of $u$ up to a many orders of magnitude higher 
level of resolution and boosts the accuracy correspondingly. 
If $\gamma<1$,
\begin{equation}    \label{eq4.10}
\|u-Qu\|_1\,\leq\,\gamma^{\,1/2}\,|\,u\,|_2,
\end{equation}
as is shown analogously to (\ref{eq3.9}). This has to be 
related to the $H^1$-distance between the exact eigenfunction 
$u$ and its approximation $\uS$, that behaves by (\ref{eq4.9}) 
like 
\begin{equation}    \label{eq4.11}
\|u\,-\uS\,\|_1\,\lesssim\,\gamma^{\,1/2}\,\varepsilon\,|\,u\,|_1.
\end{equation}
That is, one gains a factor of order $\varepsilon$, of
the size of the approximation errors (\ref{eq4.2}).
This justifies to neglect the term $\TS u-Tu$ in
equation (\ref{eq4.7}).

\pagebreak

As stated in Sect.~\ref{sec2}, the Fourier transform of the 
eigenfunction $u$ is real-analytic and all its derivatives
are bounded. Like its Fourier transform, $f=Qu$ is thus a 
rapidly decreasing function  and can therefore be expanded 
into a series of Gauss-Hermite functions or more precisely
of eigenfunctions of the harmonic oscillator that converges 
in the $H^1$-norm super-algebraically. Our goal is to show 
that, with a corresponding choice of the approximations 
(\ref{eq4.1}) of $G$ and $V$, the solution 
\begin{equation}    \label{eq4.12}  
\uS\;=\;\sum_{\nu=0}^\infty (-1)^\nu \TS^\nu\!f
\end{equation}
of the equation (\ref{eq4.8}) can be approximated up to 
given accuracy with a comparable number of Gauss functions 
as $f$, provided the width of the smoothing kernel is 
chosen sufficiently small in dependence of the 
approximation order aimed for.


\section{The core of the approximation process}
\label{sec5}

\setcounter{equation}{0}
\setcounter{lemma}{0}
\setcounter{theorem}{0}

The key to our approximation of the potential $V$ and 
of the inverse $G$ of the shifted Laplace operator 
is the approximation of the functions $1/r^{\,\beta}$, 
$\beta=1/2$ and~$1$, with very high relative accuracy
by infinite series of exponential functions. One 
obtains such expansions discretizing representations 
of $1/r^{\,\beta}$ in form of integrals over the real 
axis. We will work with a particular such construction 
due to Beylkin and Monz\'{o}n~\cite{Beylkin-Monzon}. 
Their starting point is a representation of the gamma 
function
\begin{equation}    \label{eq5.1}
\Gamma(z)\,=\,\int_0^\infty\e^{-t}t^{\,z-1}\dt,
\quad \mathrm{Re}\,z>0,
\end{equation}
in terms of an integral of a rapidly decreasing function.
\begin{lemma}       \label{lm5.1}
For all real $r>0$ and all complex $z$ with positive real part,
\begin{equation}    \label{eq5.2}
\Gamma(z)\,=\;
r^{\,z}\int_{-\infty}^\infty\!\exp(-\,r\,\e^{\,t}+z\,t)\dt.
\end{equation}
\end{lemma}
\begin{proof}
The substitution $\varphi(t)=r\,\e^{\,t}$ yields 
\begin{displaymath}
\Gamma(z)\,=\,\int_{-\infty}^\infty
\e^{-\varphi(t)}\varphi(t)^{z-1}\varphi'(t)\dt.
\end{displaymath}
Written out, this is the representation (\ref{eq5.2}).
\qed
\end{proof}
The representation (\ref{eq5.2}) of $\Gamma(z)$ leads 
conversely to the representation 
\begin{equation}    \label{eq5.3}
\frac{1}{r^{\,\beta}}\,=\,\frac{1}{\Gamma(\beta)}
\int_{-\infty}^\infty\!\exp(-\,r\,\e^{\,t}+\beta t)\dt
\end{equation}
of the function $r\to 1/r^{\,\beta}$, $r>0$, for arbitrary 
exponents $\beta>0$. The integrand decays for $t\to-\infty$ 
like $\e^{\,\beta t}$, and for $t\to\infty$ even more 
rapidly. The idea is to discretize this integral with the 
trapezoidal rule. This yields the approximation
\begin{equation}    \label{eq5.4}
\frac{1}{r^{\,\beta}}\,\approx\,\frac{1}{\Gamma(\beta)}\;
h\!\sum_{k=-\infty}^\infty\!\e^{\,\beta kh}
\exp(-\,\e^{\,kh} r)
\,=\,\frac{1}{r^{\,\beta}}\,\phi(\ln r)
\end{equation}
that depends only on the distance $h$ of the quadrature 
points. The function 
\begin{equation}    \label{eq5.5}
\phi(s)\,=\,\frac{1}{\Gamma(\beta)}\;h\!\sum_{k=-\infty}^\infty 
\exp\big(\!-\,\e^{\,kh+s}+\beta(kh+s)\big)
\end{equation}
is continuous due to the uniform convergence of the 
series on bounded intervals and periodic with period 
$h$ by definition. The relative error is therefore 
uniformly bounded in $r$ and attains its maximum 
on every interval $\e^{\,s}\leq r\leq\e^{\,s+h}$. 
The high accuracy of the approximation is a 
consequence of the following observation.
\begin{lemma}       \label{lm5.2} 
The function \rmref{eq5.5} possesses a series 
representation
\begin{equation}    \label{eq5.6}
\phi(s)\,=\,1\,+\,\frac{2}{\Gamma(\beta)}\,
\sum_{\ell=1}^\infty\, 
\bigg|\,\Gamma\bigg(\beta-\,\i\,\frac{2\pi\ell}{h}\bigg)\bigg|\,
\sin\bigg(\frac{2\pi\ell}{h}\,(s-s_\ell)\bigg),
\end{equation}
with certain phase shifts $s_\ell$. 
\end{lemma}
\begin{proof}
The proof is based on the Poisson summation formula for 
rapidly decreasing functions $f:\mathbb{R}\to\mathbb{R}$, 
more precisely on its rescaled variant
\begin{displaymath}
h\sum_{k=-\infty}^\infty\!f(kh)
\,=\,\sqrt{2\pi}\sum_{\ell=-\infty}^\infty
\fourier{f}\,\left(\frac{2\pi\ell}{h}\right).
\end{displaymath}
The Poisson summation formula is applied to 
the function
\begin{displaymath}
f(t)\,=\,\exp\big(\!-\,\e^{\,t+s}+\beta(t+s)\big)
\end{displaymath}
whose Fourier transform can be calculated with help 
of the representation of the gamma function from 
Lemma~\ref{lm5.1}. This leads to the complex 
representation
\begin{displaymath}    
\phi(s)\,=\,\frac{1}{\Gamma(\beta)}
\sum_{\ell=-\infty}^\infty\, 
\Gamma\bigg(\beta-\,\i\,\frac{2\pi\ell}{h}\bigg)
\exp\bigg(\i\,\frac{2\pi\ell}{h}\;s\bigg)
\end{displaymath}
of the function (\ref{eq5.5}). As
$\Gamma(\beta+\,\i\,\omega)$ is the complex 
conjugate of $\Gamma(\beta-\,\i\,\omega)$, 
the right-hand side is real-valued and can 
be written as above.
\qed
\end{proof}
It will turn out shortly that at least for $\beta=1/2$ 
and $\beta=1$, and not too big values of $h$, the  
term $\ell=1$ completely dominates the series; the 
other terms sum up to a remainder whose maximum norm 
is many orders of magnitude smaller than that of this 
term. A direct consequence of Lemma~\ref{lm5.2} is 
the following error estimate.
\begin{lemma}       \label{lm5.3}
For all positive real exponents $\beta$ and all $r>0$, 
\begin{equation}    \label{eq5.7}
\bigg|\;\frac{1}{r^{\,\beta}}\;-\;\frac{1}{\Gamma(\beta)}\;
h\!\sum_{k=-\infty}^\infty\!\e^{\,\beta kh}
\exp(-\,\e^{\,kh} r)\,\bigg|\;\leq\,
\frac{\varepsilon(\beta,h)}{r^{\,\beta}}
\end{equation}
holds, where the bound $\varepsilon(\beta,h)$ for the 
relative error is given by the expression
\begin{equation}    \label{eq5.8}
\varepsilon(\beta,h)\,=\,
\frac{2}{\Gamma(\beta)}\,\sum_{\ell=1}^\infty\,
\bigg|\,\Gamma\bigg(\beta-\,\i\,\frac{2\pi\ell}{h}\bigg)\bigg|.
\end{equation}
\end{lemma}
Beylkin and Monz\'{o}n~\cite{Beylkin-Monzon} estimate
the error bound $\varepsilon(\beta,h)$ for arbitrary 
positive real exponents $\beta$ and show that it tends 
exponentially to zero as $h$ goes to zero. We are 
solely interested in the exponents $\beta=1/2$ and 
$\beta=1$ and utilize that for real $\omega$
\begin{equation}    \label{eq5.9}
\bigg|\,\Gamma\bigg(\frac12-\,\i\,\omega\bigg)\bigg|^2
\,=\,\frac{\pi}{\cosh(\pi\omega)},
\quad
|\Gamma(1-\,\i\,\omega)|^2
\,=\,\frac{\pi\omega}{\sinh(\pi\omega)}
\end{equation}
holds, which is a direct consequence of Euler's 
reflection formula for the gamma function. The error 
bound $\varepsilon(1/2,h)$ possesses therefore the 
series representation
\begin{equation}    \label{eq5.10}
\varepsilon\bigg(\frac12,\,h\bigg)=\;
2\,\sqrt{2}\;\sum_{\ell=1}^\infty\frac{q^{\,\ell}\;}{(1+q^{\,4\ell})^{1/2}}
\end{equation}
in terms of the variable $q=\e^{-\pi^2/h}$, and the 
error bound $\varepsilon(1,h)$ the expansion
\begin{equation}    \label{eq5.11}
\varepsilon(1,h)\,=\;4\pi\,h^{-1/2}
\sum_{\ell=1}^\infty\frac{\sqrt{\ell}\,q^{\,\ell}\;}{(1-q^{\,4\ell})^{1/2}}.
\end{equation}
We conclude that the maximum norm of the relative error
itself as well as its upper bound $\epsilon(\beta,h)$  
differ for $\beta=1/2$ and $\beta=1$, respectively, from 
the quantities
\begin{equation}    \label{eq5.12}
2\,\sqrt{2}\;\e^{-\pi^2/h},\quad 4\pi\,h^{-1/2}\e^{-\pi^2/h}
\end{equation}
only by factors that tend themselves like $\sim\e^{-\pi^2/h}$ 
to one as $h$ goes to zero. The relative error thus tends  
very rapidly to zero. For both cases, it is already less than 
$10^{-7}$ for $h=1/2$, less than $10^{-15}$ for $h=1/4$, and 
less than $10^{-32}$ for $h=1/8$.


\section{The approximations of the operators}
\label{sec6}

\setcounter{equation}{0}
\setcounter{lemma}{0}
\setcounter{theorem}{0}

The approximations (\ref{eq5.4}) of $1/\sqrt{r}$ and 
$1/r$ lead to the approximations
\begin{equation}    \label{eq6.1}
\frac{1}{r}\,\approx\,\frac{1}{\sqrt{\pi}}\;
h\!\sum_{k=-\infty}^\infty\!\e^{\,kh/2}
\exp(-\,\e^{\,kh} r^{\,2}),
\quad
\frac{1}{r}\,\approx\;
h\!\sum_{k=-\infty}^\infty\!\e^{\,kh}
\exp(-\,\e^{\,kh} r)
\end{equation}
of $1/r$ by sums of Gauss and exponential functions, 
respectively. These form the basis of our approximations 
(\ref{eq4.1}) of the interaction potentials between the 
electrons and the electrons and nuclei and of the inverse 
(\ref{eq3.1}):
\begin{equation}    \label{eq6.2}
(Gf)(\x)\,=\,\Big(\frac{1}{\sqrt{2\pi}}\Big)^{3N}\!
\int \frac{1}{\,|\vomega|^2-\,\lambda}\;
\fourier{f}\,(\vomega)\,\e^{\,\i\,\vomega\,\cdot\,\x}\domega
\end{equation}
of the shifted Laplace operator $-\Delta-\lambda$.
It is approximated replacing its symbol 
\begin{equation}    \label{eq6.3}
\frac{1}{\,|\vomega|^2-\,\lambda}
\,\approx\;
h\!\sum_{k=-\infty}^\infty\!\e^{\,kh}
\exp\big(\!-\,\e^{\,kh}(|\vomega|^2-\,\lambda)\big)
\end{equation}
by the second of the approximations from (\ref{eq6.1}). 
That is, we replace $G$ by the sum 
\begin{equation}    \label{eq6.4}
\GS\;=\sum_{k=-\infty}^\infty\!G_k
\end{equation}
of the operators $G_k:L_2\to H^2$ with symbols
\begin{equation}    \label{eq6.5}
h\;\exp(\e^{\,kh}\lambda+kh)\exp({-\,\e^{\,kh}}|\vomega|^2).
\end{equation}
The pointwise convergence of the series (\ref{eq6.4}) 
to its limit $\GS:L_2\to H^2$ can be shown using 
(\ref{eq5.7}) and the dominated convergence theorem. 
Its norm convergence with respect to an appropriate
scale of norms will follow from the considerations 
in the next section. For square integrable functions $f$ 
\begin{equation}    \label{eq6.6}
\|Gf-\GS f\|_1\,\leq\,\varepsilon(1,h)\|Gf\|_1
\end{equation}
holds, as follows from Lemma~\ref{lm5.3}.
The interaction potential 
\begin{equation}    \label{eq6.7}
V(\x)\;=\;-\;\sumi\sum_{\nu=1}^K\frac{Z_\nu}{|\x_i-\a_\nu|}
\,+\,\frac12 \sumij\frac{1}{|\x_i-\x_j|}
\vspace{-0.5\belowdisplayskip}
\end{equation}
is treated correspondingly replacing it by the sum
\begin{equation}    \label{eq6.8}
\VS\;=\sum_{k=-\infty}^\infty\!V_k
\vspace{-0.5\belowdisplayskip}
\end{equation}
of the parts given by the expression
\begin{equation}    \label{eq6.9}
V_k(\x)\;=\;-\;\sumi\sum_{\nu=1}^K\,Z_\nu\,\phi_k(\x_i-\a_\nu)
\,+\,\frac12 \sumij\phi_k(\x_i-\x_j),
\vspace{-0.5\belowdisplayskip}
\end{equation}
where the $\phi_k:\mathbb{R}^3\to\mathbb{R}$ are 
the Gauss functions 
\begin{equation}    \label{eq6.10}
\phi_k(\x)\,=\,
\frac{1}{\sqrt{\pi}}\;h\;\e^{\,kh/2}\exp(-\,\e^{\,kh}\,|\x|^2).
\end{equation}
The approximation error for the single interaction terms 
can again be estimated with the help of the estimate 
(\ref{eq5.7}) from Lemma~\ref{lm5.3}. 
The three-dimensional Hardy inequality (\ref{eq2.4}) 
and some elementary calculations using the triangle and 
the Cauchy-Schwarz inequality lead to the error estimate
\begin{equation}    \label{eq6.11}
\|Vu-\VS u\|_0\,\leq\,
\theta(N,Z)\,\varepsilon(1/2,h)|\,u\,|_1
\end{equation}
for functions $u\in H^1$, where the prefactor (\ref{eq3.11}):
\begin{equation}    \label{eq6.12}
\theta(N,Z)\,=\,(2\,Z+N-1)\,N^{\,1/2}
\end{equation}
is that from Lemma~\ref{lm2.1} and $Z$ is the total 
charge of the nuclei.
We conclude that already for a moderate choice of the 
distance $h$ between the quadrature points in (\ref{eq5.4})
and (\ref{eq6.1}), respectively, we get approximations of 
extremely high accuracy of the inverse $G$ of the shifted 
Laplace operator and of the multiplication operator $V$. 
The choice of the gridsize $h$ fixes the substitute problem 
(\ref{eq4.8}) that replaces the original equation 
(\ref{eq3.8}) and (\ref{eq4.7}), respectively, and whose 
solution we want to approximate in what follows instead 
of the given eigenfunction.

The crucial point is that the single parts of which the
approximate operators are composed map Gauss functions
to Gauss functions, pure Gauss functions to pure Gauss
functions and Gauss-Hermite functions to Gauss-Hermite
functions with a polynomial part of same degree.
Let us discuss at first the case of pure Gauss function,
without polynomial part. The multiplication of such a 
Gauss function with one of the parts $V_k$ yields a 
finite linear combination of such Gauss functions.  
The reason is that the product of two such Gauss 
functions
\begin{equation}    \label{eq6.13}    
\exp\left(-\frac12\,(\x-\a_1)\cdot\Q_1(\x-\a_1)\right),
\quad
\exp\left(-\frac12\,(\x-\a_2)\cdot\Q_2(\x-\a_2)\right)
\end{equation}
is a scalar multiple of the Gauss function
\begin{equation}    \label{eq6.14}    
\exp\left(-\frac12\,(\x-\a)\cdot\Q(\x-\a)\right)
\end{equation}
with matrix and shift 
\begin{equation}    \label{eq6.15}
\Q\,=\,\Q_1+\Q_2, \quad
\a\,=\,(\Q_1+\Q_2)^{-1}(\Q_1\a_1+\Q_2\a_2).
\end{equation}
Or consider the application of one of the operators 
$G_k$ or of the operator (\ref{eq3.7}) to a Gauss 
function (\ref{eq6.14}). The Fourier transform of 
this Gauss function is 
\begin{equation}    \label{eq6.16}    
\frac{1}{\sqrt{\det\Q}}\;\,\e^{-\i\,\a\,\cdot\,\vomega}\,
\exp\left(-\frac12\,\vomega\cdot\Q^{-1}\vomega\right).
\end{equation}
Multiplication with an isotropic Gauss function
\begin{equation}    \label{eq6.17}
\exp\left(-\frac12\,\alpha \,|\vomega|^2\right),
\quad \alpha\,>\,0,
\end{equation}
and application of the inverse Fourier transform 
lead to the Gauss function
\begin{equation}    \label{eq6.18}    
\frac{1}{\sqrt{\det(\I+\alpha\Q)}}\;
\exp\left(-\frac12\,(\x-\a)\cdot(\I+\alpha\Q)^{-1}\Q(\x-\a)\right).
\end{equation}
If the matrices under consideration are of the form 
$\Q=\Q\,'\otimes\,\I_{\,3}$, with symmetric-positive 
definite matrices $\Q\,'$ of dimension $N$ and the 
$(3\times 3)$-identity matrix $\I_{\,3}$, none of the 
operations leads out of this class, so that all
computations can be reduced to computations in the
$N$-dimensional space.

The argumentation for Gauss-Hermite functions is similar.
This is obvious for the multiplication with one of the
parts $V_k$ because this does not affect the polynomial
factor. The rest follows from the fact that there is
a direct correspondence between the multiplication of a 
function with a polynomial and the derivatives of its 
Fourier transform. For all multi-indices $\valpha$ and 
all rapidly decreasing functions $v$,
\begin{equation}    \label{eq6.19}
\mathcal{F}(\x^\valpha v)\,=\,
\i^{\,|\valpha|}\mathrm{D}^\valpha\mathcal{F}v,
\quad
\mathcal{F}^{-1}(\vomega^\valpha v)\,=\,
(-\i)^{|\valpha|}\mathrm{D}^\valpha\mathcal{F}^{-1}v.
\end{equation}
If $v$ is a Gauss function (now without polynomial part)
centered around the origin, the derivatives 
$\mathrm{D}^\valpha\mathcal{F}v$ and 
$\mathrm{D}^\valpha\mathcal{F}^{-1}v$ of the Fourier 
transform and of the inverse Fourier transform of $v$ are 
products of the Gauss functions $\mathcal{F}v$ and 
$\mathcal{F}^{-1}v$ with polynomials of degree $|\valpha|$,
so that one remains in the given class. A shift of the
center is easily included and does not affect this property.
The hope is therefore that the solution of our substitute 
equation, the function 
\begin{equation}    \label{eq6.20}
\uS\;=\;\sum_{\nu=0}^\infty (-1)^\nu \TS^\nu\!f,
\end{equation}
can be well approximated by linear combinations of Gauss 
functions as long as this holds for the right hand side
$f$ of this equation.


\section{Estimates of the norms of the single components}
\label{sec7}

\setcounter{equation}{0}
\setcounter{lemma}{0}
\setcounter{theorem}{0}

To obtain approximations of the solution (\ref{eq4.12}), 
(\ref{eq6.20}) of the substitute equation (\ref{eq4.8}) by 
finite linear combinations of Gauss functions, we need to 
truncate both the series on the right hand side of 
(\ref{eq6.20}) itself and the series of Gauss functions 
representing the single terms $\TS^\nu\!f$. For this purpose 
we utilize estimates in fractional order Sobolev spaces 
$H^\vartheta$ for the norms of the single parts of which the 
approximate operators (\ref{eq6.4}) and (\ref{eq6.8}) are 
composed. We will show in this section that the norms of 
these single parts tend exponentially to zero as $k$ goes 
to plus or minus infinity.

The Sobolev space $H^\vartheta$, $\vartheta$ an arbitrary 
real number, is the completion of the space of the 
real-valued, rapidly decreasing functions, or even the 
infinitely differentiable functions with compact support, 
under the norm given by the expression
\begin{equation}    \label{eq7.1}
\|u\|_\vartheta^2\,=\,
\int\big(1+|\vomega|^2\big){}^\vartheta\,|\fourier{u}(\vomega)|^2\domega.
\end{equation}
For $\vartheta=0$, this is the $L_2$-norm and the corresponding 
Sobolev space $H^\vartheta$ is $L_2$, and for $\vartheta=1$, 
the norm coincides with the standard norm on $H^1$. If $\vartheta$ 
is an integer greater than zero, $H^\vartheta$ consists of the 
$\vartheta$-times weakly differentiable functions with weak 
derivatives in $L_2$. For values $\vartheta>0$, we use also the 
seminorm on $H^\vartheta$ given by
\begin{equation}    \label{eq7.2}
|\,u\,|_\vartheta^2\,=\,
\int|\vomega|^{2\vartheta}\,|\fourier{u}(\vomega)|^2\domega.
\end{equation}
For $\vartheta=1$, this seminorm is again the usual 
seminorm on $H^1$.

We begin with the parts of which the interaction potentials
are composed. Starting point is the estimate from the next 
lemma for functions $u:\mathbb{R}^3\to\mathbb{R}$ that 
is based on the three-dimensional Hardy-Rellich inequality. 
For $u\in H^\vartheta$, $0<\vartheta<3/2$,
\begin{equation}    \label{eq7.3}     
\int\!\frac{1}{\,|\x|^{2\vartheta}}\;|u(\x)|^2\dx \;\leq\;
\frac{4^\vartheta}{\min(1,(3-2\vartheta)^2)}\,
\int |\vomega|^{2\vartheta}\,|\fourier{u}(\vomega)|^2\domega.
\end{equation}
For $\vartheta=1$, this is the classical Hardy inequality 
(\ref{eq2.4}). For $\vartheta<1$, the inequality can be 
derived by interpolation from the classical Hardy inequality. 
For \mbox{$1<\vartheta<3/2$}, the inequality is at first 
reduced, similarly as in the proof from \cite{Ys_2010} of 
the classical Hardy inequality, to an estimate of the given 
weighted norm of $u$ by a weighted norm of $\nabla u$ that 
is then further estimated using the inequality for the 
already known case $\vartheta<1$. The optimal constant can 
be calculated with much more effort expanding the functions 
under consideration into products of radial parts and 
spherical harmonics; see \cite{Yafaev}. It behaves like 
$\sim 1/(3-2\vartheta)^2$ when~$\vartheta$ approaches the 
limit value $3/2$, a behavior that reflects the estimate
(\ref{eq7.3}) properly. 
\begin{lemma}       \label{lm7.1}
For all indices $0<\vartheta<1/2$, all nonnegative integers $k$, 
and all rapidly decreasing functions $u:\mathbb{R}^3\to\mathbb{R}$,
the estimate
\begin{equation}    \label{eq7.4}
\|\phi_k u\|_0\,\leq\,\kappa(\vartheta)\,\frac{\vartheta h}{2}\,
\exp\left(-\,\frac{\vartheta h}{2}\,|\,k\,|\,\right)
|\,u\,|_{1+\vartheta}
\end{equation}
holds, where the $\phi_k$ are given by \rmref{eq6.10} and
the constant $\kappa(\vartheta)$ is defined by 
\begin{equation}    \label{eq7.5}
\kappa(\vartheta)\,=\,
\frac{1}{\sqrt{\pi}}\,
\left(\frac{2+2\vartheta}{\e}\right)^{\!(1+\vartheta)/2}\!\!
\frac{2}{\vartheta\,(1-2\vartheta)}.
\end{equation}
For negative integers $k$, the following estimate holds:
\begin{equation}    \label{eq7.6}                
\|\phi_k u\|_0\,\leq\,\kappa(\vartheta)\,\frac{\vartheta h}{2}\,
\exp\left(-\,\frac{\vartheta h}{2}\,|\,k\,|\,\right)
|\,u\,|_{1-\vartheta}.
\end{equation}
\end{lemma}
\begin{proof}
The square of the $L_2$-norm to be estimated can be written as
\begin{displaymath}
\|\phi_k u\|_0^2\;=\;
\frac{h^2}{\pi\,}\;\e^{-\vartheta kh}
\int\big(\e^{\,kh}\,|\x|^2\big){}^{1+\vartheta}
\exp\big(\!-2\,\e^{\,kh}\,|\x|^2\big)\,
\frac{1}{\,|\x|^{2\,(1+\vartheta)}}\;|u(\x)|^2\dx.
\end{displaymath}
Since the expression $t^{1+\vartheta}\e^{-2t}$, $t\geq 0$,
attains its maximum at $t=(1+\vartheta)/2$,
\begin{displaymath}
\|\phi_k u\|_0^2\;\leq\;
\frac{h^2}{\pi\,}\,
\left(\frac{1+\vartheta}{2\,\e}\right)^{\!1+\vartheta}
\!\e^{-\vartheta kh}
\int\frac{1}{\,|\x|^{2\,(1+\vartheta)}}\;|u(\x)|^2\dx
\end{displaymath}
follows. The Hardy inequality (\ref{eq7.3}) yields
\begin{displaymath}
\|\phi_k u\|_0^2\;\leq\;
\frac{h^2}{\pi\,}\,
\left(\frac{1+\vartheta}{2\,\e}\right)^{\!1+\vartheta}\!\!
\frac{4^{1+\vartheta}}{(1-2\vartheta)^2}\;\,
\e^{-\vartheta kh}\,|\,u\,|_{1+\vartheta}^2.
\end{displaymath}
For $k\geq 0$, this is the estimate (\ref{eq7.4}). 
For $k<0$, one starts from the representation
\begin{displaymath}
\|\phi_k u\|_0^2\;=\;
\frac{h^2}{\pi\,}\;\e^{\,\vartheta kh}
\int\big(\e^{\,kh}\,|\x|^2\big){}^{1-\vartheta}
\exp\big(\!-2\,\e^{\,kh}\,|\x|^2\big)\,
\frac{1}{\,|\x|^{2\,(1-\vartheta)}}\;|u(\x)|^2\dx
\end{displaymath}
and obtains correspondingly the estimate 
\begin{displaymath}
\|\phi_k u\|_0^2\;\leq\;
\frac{h^2}{\pi\,}\,
\left(\frac{1-\vartheta}{2\,\e}\right)^{\!1-\vartheta}
\!\!4^{1-\vartheta}\,
\e^{\,\vartheta kh}\,|\,u\,|_{1-\vartheta}^2.
\end{displaymath}
Since for all $\vartheta$ in the interval under consideration
\begin{displaymath}
\left(\frac{2-2\vartheta}{\e}\right)^{\!1-\vartheta}
\!\leq\;
\left(\frac{2+2\vartheta}{\e}\right)^{\!1+\vartheta}
\!\!\frac{1}{(1-2\vartheta)^2}
\end{displaymath}
holds, this proves the estimate (\ref{eq7.6}) for the 
case of negative integers $k$.
\qed
\end{proof}
Next we transfer these estimates to the multi-particle 
case and consider functions
\begin{equation}    \label{eq7.7}
u:(\mathbb{R}^3)^N\to\;\mathbb{R}:
(\x_1,\ldots,\x_N)\;\to\;u(\x_1,\ldots,\x_N).
\end{equation}
\begin{lemma}       \label{lm7.2}
For $0<\vartheta<1/2$, for all nonnegative integers $k$, for
all indices $i\neq j$, and for all rapidly decreasing 
functions $u:\mathbb{R}^{3N}\to\mathbb{R}$,
\begin{equation}    \label{eq7.8}
\int|\phi_k(\x_i-\x_j)\,u(\x)|^2\dx 
\;\leq\; K^{\,2}
\int|\vomega_i|^{2\,(1+\vartheta)}|\fourier{u}(\vomega)|^2\domega,
\end{equation}
where $K$ is here an abbreviation for the expression
\begin{equation}    \label{eq7.9}
K\,=\,\kappa(\vartheta)\,\frac{\vartheta h}{2}\,
\exp\left(-\,\frac{\vartheta h}{2}\,|\,k\,|\,\right).
\end{equation}
For negative integers $k$, the following estimate holds:
\begin{equation}    \label{eq7.10}
\int|\phi_k(\x_i-\x_j)\,u(\x)|^2\dx 
\;\leq\; K^{\,2}
\int|\vomega_i|^{2\,(1-\vartheta)}|\fourier{u}(\vomega)|^2\domega.
\end{equation}
\end{lemma}
\begin{proof}
We split the vectors $\x=(\x_i,\x')$ in $\mathbb{R}^{3N}$ into 
the component $\x_i$ in $\mathbb{R}^3$ and the remaining part 
$\x'$. As, for any given $\x'$, $\x_i\to u(\x_i,\x')$ is a  
rapidly decreasing function from $\mathbb{R}^3$ to $\mathbb{R}$ 
and as the seminorms (\ref{eq7.2}) are shift-invariant, 
by Lemma~\ref{lm7.1} then
\begin{displaymath}
\int|\phi_k(\x_i-\x_j)\,u(\x_i,\x')|^2\,\diff{\x_i}
\;\leq\;K^{\,2}\int
|\vomega_i|^{2\,(1+\vartheta)}|(\mathcal{F}_i u)(\vomega_i,\x')|^2\,
\diff{\vomega_i},
\end{displaymath}
where $\mathcal{F}_i$ denotes the Fourier transform with respect 
to $\x_i$. Correspondingly, let $\mathcal{F}'$ be the Fourier 
transform with respect to the remaining variables $\x'$. 
Integration of this inequality with respect to the remaining 
variables, Fubini's theorem, that
\begin{displaymath}
\int|(\mathcal{F}_i u)(\vomega_i,\x')|^2\,\diff{\x'}
\,=\,
\int|(\mathcal{F}'\!\mathcal{F}_i u)(\vomega_i,\vomega')|^2\,
\diff{\vomega'}
\end{displaymath}
by Plancherel's theorem, and  the observation that 
$\mathcal{F}'\!\mathcal{F}_i$ is the Fourier transform with 
respect to the full set of variables, finally yield the 
estimate (\ref{eq7.8}). The estimate (\ref{eq7.10}) for
the case of negative integers $k$ is proved in the
same way. 
\qed
\end{proof}
Estimates of the same type hold for the terms coming from the 
interaction between the electrons and the nuclei and can be
derived in the same way.
\begin{lemma}       \label{lm7.3}
For $0\!<\!\vartheta\!<\!1/2$, for all nonnegative integers $k$, 
for all indices $i$ and~$\nu$, and for all rapidly decreasing
functions $u:\mathbb{R}^{3N}\to\mathbb{R}$,
\begin{equation}    \label{eq7.11}
\int|\phi_k(\x_i-\a_\nu)\,u(\x)|^2\dx 
\;\leq\; K^{\,2}
\int|\vomega_i|^{2\,(1+\vartheta)}|\fourier{u}(\vomega)|^2\domega,
\end{equation}
where $K$ is again the abbreviation for the expression
\begin{equation}    \label{eq7.12}
K\,=\,\kappa(\vartheta)\,\frac{\vartheta h}{2}\,
\exp\left(-\,\frac{\vartheta h}{2}\,|\,k\,|\,\right).
\end{equation}
For negative integers $k$, the following estimate holds:
\begin{equation}    \label{eq7.13}
\int|\phi_k(\x_i-\a_\nu)\,u(\x)|^2\dx 
\;\leq\; K^{\,2}
\int|\vomega_i|^{2\,(1-\vartheta)}|\fourier{u}(\vomega)|^2\domega.
\end{equation}
\end{lemma}
We finally combine the estimates from Lemma~\ref{lm7.2} 
and Lemma~\ref{lm7.3} for its parts to estimates for
the functions (\ref{eq6.9}), interpreted as a 
multiplication operators mapping the functions in
the space $H^{1+\vartheta}$ to functions in $L_2$.
\begin{lemma}       \label{lm7.4}
For all indices $0\!<\!\vartheta\!<\!1/2$, for all 
negative integers $k$, and for all rapidly decreasing 
functions $u:\mathbb{R}^{3N}\to\mathbb{R}$,
\begin{equation}    \label{eq7.14}
\|V_ku\|_0\;\leq\;
\frac{(2\,Z+N-1)\,N^{\,(1+\vartheta)/2}}{2}\;
\kappa(\vartheta)\,\frac{\vartheta h}{2}\,
\exp\left(-\,\frac{\vartheta h}{2}\,|\,k\,|\,\right)
|\,u\,|_{1-\vartheta}.
\end{equation}
For nonnegative integers $k$, the following estimate holds:
\begin{equation}    \label{eq7.15}
\|V_ku\|_0\;\leq\;
\frac{(2\,Z+N-1)\,N^{\,1/2}}{2}\;
\kappa(\vartheta)\,\frac{\vartheta h}{2}\,
\exp\left(-\,\frac{\vartheta h}{2}\,|\,k\,|\,\right)
|\,u\,|_{1+\vartheta}.
\end{equation}
As before, $Z=\sum_\nu Z_\nu$ is here again the total 
charge of the nuclei.
\end{lemma}
\begin{proof}
For negative integers $k$, the estimates (\ref{eq7.10}) and 
(\ref{eq7.13}) for the single parts of~$V_k$ and the triangle 
and the Cauchy-Schwarz inequality at first yield 
\begin{displaymath}  
\|V_ku\|_0\;\leq\;\frac{(2\,Z+N-1)\,N^{\,1/2}}{2}\;
\bigg(K^{\,2}\,\sum_{i=1}^N
\int|\vomega_i|^{2\,(1-\vartheta)}|\fourier{u}(\vomega)|^2\domega
\bigg)^{1/2},
\end{displaymath}
where the abbreviation (\ref{eq7.9}) has again been used.
By H\"older's inequality
\begin{displaymath}
\sum_{i=1}^N \eta_i^{2\,(1-\vartheta)}\leq\;N^\vartheta
\end{displaymath}
for all $\veta$ on the surface of the $N$-dimensional unit 
sphere, from which (\ref{eq7.14}) follows. The case of 
nonnegative integers $k$ is treated analogously, where 
the estimate
\begin{displaymath}
\sum_{i=1}^N \eta_i^{2\,(1+\vartheta)}\leq\;1
\end{displaymath}
for the $\veta$ on the surface of the $N$-dimensional unit 
sphere enters.
\qed 
\end{proof}
Because of $|\,u\,|_{1\pm \vartheta}\leq\|u\|_{1+\vartheta}$,
the two estimates can be combined to the estimate
\begin{equation}    \label{eq7.16}
\|V_ku\|_0\;\leq\;
\frac{(2\,Z+N-1)\,N^{\,(1+\vartheta)/2}}{2}\;
\kappa(\vartheta)\,\frac{\vartheta h}{2}\,
\exp\left(-\,\frac{\vartheta h}{2}\,|\,k\,|\,\right)
\|u\|_{1+\vartheta}
\end{equation}
that holds both for positive and negative integers $k$.

Next we estimate the norms of the operators $G_k$ of which 
the approximate inverse (\ref{eq6.4}) of the shifted 
Laplacian $-\Delta-\lambda$ is composed. We interpret these 
operators as operators from $L_2$ to the spaces 
$H^{\,2-\vartheta}$ for indices $\vartheta$ between $0$ and $2$.
\begin{lemma}       \label{lm7.5}
For all indices $0<\vartheta<2$, for all integers $k$, 
and for all rapidly decreasing functions 
$f:\mathbb{R}^{3N}\to\mathbb{R}$ the following 
estimate holds:
\begin{equation}    \label{eq7.17}
|\,G_k f\,|_{2-\vartheta}\;\leq\;
h\;\left(\frac{2-\vartheta}{2\,\e}\right)^{\!(2-\vartheta)/2}\!\!
\!\exp\left(\e^{\,kh}\lambda+\frac{\vartheta}{2}\,kh\right)\|f\|_0.
\end{equation}
\end{lemma}
\begin{proof}
We rewrite the square of the left hand side of
(\ref{eq7.17}) at first in the form
\begin{displaymath}
|\,G_k f\,|_{2-\vartheta}^2=\;h^2\,
\exp\big(2\,\e^{\,kh}\lambda+\vartheta kh\big)
\int\!
\big(\e^{\,kh}|\vomega|^2\big){}^{2-\vartheta}
\exp\big(\!-2\,\e^{\,kh}|\vomega|^2\big)
|\fourier{f}\,(\vomega)|^2\domega.
\end{displaymath}
Using that the expression $t^{2-\vartheta}\e^{-2\,t}$, $t>0$,
attains its maximum at $t=(2-\vartheta)/2$,
\begin{displaymath}
|\,G_k f\,|_{2-\vartheta}^2\,\leq\;\,h^2\,
\left(\frac{2-\vartheta}{2\,\e}\right)^{\!2-\vartheta}\!\!
\exp\big(2\,\e^{\,kh}\lambda+\vartheta kh\big)
\int|\fourier{f}\,(\vomega)|^2\domega
\end{displaymath}
follows, which was the proposition.
\qed
\end{proof}
The conclusion is that, for $0<\vartheta<2$, the norms 
of the operators 
\begin{equation}    \label{eq7.18}
G_k:L_2\to H^{\,2-\vartheta} 
\end{equation}
tend again rapidly to zero as $k$ goes to plus or minus 
infinity, exponentially as $k$ goes to minus infinity, 
and super exponentially as $k$ goes to plus infinity.
Because of
\begin{displaymath}
\exp\left(\e^{\,t}\lambda+\frac{\vartheta}{2}\;t\right)
\,\leq\;\max\left(1\,,\;
\Big(\!-\frac{\vartheta}{\lambda\e}\,\Big)^{\!\vartheta}\,\right)
\exp\left(-\,\frac{\vartheta}{2}\,|\,t\,|\,\right)
\end{displaymath} 
one obtains again, with the correspondingly chosen constant
$\kappa^*(\lambda,\vartheta)$, the in comparison to the
original estimate (\ref{eq7.17}) for later purposes more 
convenient estimate
\begin{equation}    \label{eq7.19}
|\,G_k f\,|_{2-\vartheta}\;\leq\;
\kappa^*(\lambda,\vartheta)\,\frac{\vartheta h}{2}\,
\exp\left(-\,\frac{\vartheta h}{2}\,|\,k\,|\,\right)\|f\|_0
\end{equation}
for $f\in L_2$, 
which, however, severely overestimates the norms 
for positive $k$.   

Central for our argumentation is the splitting of the 
potential term $Vu$ into the smooth part $QVu$, the 
convolution of $Vu$ with the Gaussian kernel (\ref{eq3.5}),
and the complementary part $PVu$. The cut-off operator $P$
filtering out the low-frequency part of $Vu$ 
will make sure that the norms of the combined operators 
remain sufficiently small and do not exceed certain 
bounds. It reads in terms of the Fourier transform
\begin{equation}    \label{eq7.20}
\widehat{Pv\,}(\vomega)\,=\,
\big(1-\e^{-\gamma\,|\vomega|^2}\big)\,\fourier{v}(\vomega).
\end{equation}
The constant $\gamma<1$ determines how strongly low  
frequencies are damped and how much $P$ reduces norms 
in the sense of the following estimate.
\begin{lemma}       \label{lm7.6}
For all indices $0<\vartheta<1/2$, all constants 
$\gamma<1$, and all rapidly decreasing functions
$v:\mathbb{R}^{3N}\to\mathbb{R}$, the following
estimate holds:
\begin{equation}    \label{eq7.21}
\|Pv\|_{1+\vartheta}\,\leq\,
\sqrt{2}\,\gamma^{\,1/2-\vartheta}|\,v\,|_{2-\vartheta}.
\end{equation}
\end{lemma}
\begin{proof}
The square of the norm of $Pv$ can be written as
\begin{displaymath}
\|Pv\|_{1+\vartheta}^2\,=\,\gamma^{\,1-2\vartheta}\!
\int a\big(\vartheta,\gamma,\,\gamma\,|\vomega|^2\big)
|\vomega|^{2\,(2-\vartheta)}|\fourier{v}(\vomega)|^2\domega,
\end{displaymath}
where the function $a(\vartheta,\gamma,\,t)$ is given by
\begin{displaymath}
a(\vartheta,\gamma,\,t)\,=\;t^\vartheta (\gamma+t)^{1+\vartheta}
\left(\frac{1-\,\e^{-t}}{t}\right)^2
\end{displaymath}
and can for the given $\gamma$ and $\vartheta$ be 
roughly estimated as follows:
\begin{displaymath}
a(\vartheta,\gamma,\,t)\,\leq\;
(1+t)^2\left(\frac{1-\,\e^{-t}}{t}\right)^2.
\end{displaymath}
To estimate the bound on the right hand side further, let
\begin{displaymath}
f(t)\,=\,\mu t-(1+t)(1-\e^{-t}), \quad \mu=1+\e^{-1}.
\end{displaymath}
Since $f\,''(t)=-\,(1-t)\,\e^{-t}$, the derivative of $f$ 
takes its minimum $f\,'(1)=0$ at $t=1$. That is, $f(t)$
increases strictly and it is $f(t)\geq f(0)=\,0$ for 
all $t\geq 0$. Therefore
\begin{displaymath}
a(\vartheta,\gamma,\,t)\,\leq\;\mu^2.
\end{displaymath}
Because $\e\geq 5/2$, $\mu^2\leq 2$ and the 
estimate (\ref{eq7.21}) is proven.
\qed
\end{proof}

We could stop our considerations here and could 
as in \cite{Scholz} continue with the given estimate 
(\ref{eq7.16}) for the $V_k$, seen as operators from 
$H^{1+\vartheta}$ to $L_2$, and the estimate 
resulting from (\ref{eq7.19}) and (\ref{eq7.21}) 
for the operators $G_kP=PG_k$ from $L_2$ back 
to $H^{1+\vartheta}$. This would result in an 
analysis on the approximation of the wavefunctions 
in a space $H^{1+\vartheta}$ for some $\vartheta$ 
between $0$ and $1/2$. Since we are, however, 
finally interested in the approximation properties 
in the energy space underlying the electronic 
Schr\"odinger equation, the Sobolev space $H^1$, 
we shift our estimates still downward by the 
chosen~$\vartheta$ using a simple duality and 
interpolation argument.
\begin{lemma}       \label{lm7.7}
The multiplication operators $V_k$ can be uniquely 
extended to bounded linear operators from $H^1$ to 
the dual space $H^{-\vartheta}$, $0<\vartheta<1/2$. 
For all $u\in H^1$,
\begin{equation}    \label{eq7.22}
\|V_ku\|_{-\vartheta}\;\leq\;
\frac{(2\,Z+N-1)\,N^{\,(1+\vartheta)/2}}{2}\;
\kappa(\vartheta)\,\frac{\vartheta h}{2}\,
\exp\left(-\,\frac{\vartheta h}{2}\,|\,k\,|\,\right)
\|u\|_1.
\end{equation}
\end{lemma}
\begin{proof}
It suffices to prove the estimate for rapidly decreasing 
functions $u$ as these are dense in the spaces under 
consideration. Because $V_k$ maps rapidly decreasing 
functions to rapidly decreasing functions, this considerably
simplifies the argumentation. We start from the estimate 
(\ref{eq7.16}), here written as
\begin{displaymath}
\|V_ku\|_0\,\leq\, c\,\|u\|_{1+\vartheta}.
\end{displaymath}
It implies that for all rapidly decreasing functions 
$\varphi$
\begin{displaymath}
(V_ku,\varphi)\,=\,(u,V_k\varphi)\,\leq\, 
c\,\|u\|_0\|\varphi\|_{1+\vartheta}.
\end{displaymath}
A rapidly decreasing function $\varphi$ is real-valued 
if and only if $\overline{\fourier{\varphi}(\vomega)}=
\fourier{\varphi}(-\vomega)$ holds. Inserting the thus 
real-valued rapidly decreasing function $\varphi$ with 
Fourier transform
\begin{displaymath}
\fourier{\varphi}(\vomega)\,=\,
\big(1+|\vomega|^2\big){}^{-(1+\vartheta)}\,
\fourier{v}(\vomega),
\quad  v\,=\,V_ku,
\end{displaymath}
this leads by means of Plancherel's theorem to 
the dual estimate
\begin{displaymath}
\|V_ku\|_{-(1+\vartheta)}\,\leq\, c\,\|u\|_0.
\end{displaymath}
The estimate (\ref{eq7.22}) and generally for
$0\leq s\leq 1+\vartheta$ and $t=s-(1+\vartheta)$ 
the estimate
\begin{displaymath}
\|V_ku\|_{\,t}\,\leq\, c\,\|u\|_{\,s}
\end{displaymath}
follow from the original and the dual estimate 
by interpolation within the space of the rapidly 
decreasing functions; details can be found in 
the appendix.
\qed
\end{proof}
For the composed operators $G_kP$ we proceed in the
same way.
\begin{lemma}       \label{lm7.8}
The operators $G_kP$ can be uniquely extended to bounded 
linear operators from the dual space $H^{-\vartheta}$,
$0<\vartheta<1/2$, to $H^1$. For all $f\in H^{-\vartheta}$,
\begin{equation}    \label{eq7.23}
\|G_kPf\|_1\;\leq\;
\sqrt{2}\,\kappa^*(\lambda,\vartheta)\,\frac{\vartheta h}{2}\,
\exp\left(-\,\frac{\vartheta h}{2}\,|\,k\,|\,\right)
\gamma^{\,1/2-\vartheta}\|f\|_{-\vartheta}.
\end{equation}
\end{lemma}
\begin{proof}
We take again advantage of the fact that it suffices 
to prove the estimate for rapidly decreasing functions 
$f$ and that $G_kP$ maps rapidly decreasing functions
to rapidly decreasing functions.  The fact that the 
operators $G_k$ and $P$ commute, the estimate 
(\ref{eq7.21}) for the norm of the cut-off operator 
$P$, and the estimate (\ref{eq7.19}) for the norm of 
the $G_k$ lead at first to the estimate
\begin{displaymath}
\|G_kPf\|_{1+\vartheta}\,\leq\,c\,\|f\|_0,
\end{displaymath}
where $c$ denotes here the constant from (\ref{eq7.23}).
For all rapidly decreasing functions $\varphi$ by 
Plancherel's theorem and the Cauchy-Schwarz inequality 
therefore
\begin{displaymath}
(G_kPf,\varphi)\,=\,(f,G_kP\varphi)\,\leq\,
c\,\|f\|_{-(1+\vartheta)}\|\varphi\|_0.
\end{displaymath}
Inserting the rapidly decreasing function 
$\varphi=G_kPf$, this yields the dual estimate
\begin{displaymath}
\|G_kPf\|_0\,\leq\,c\,\|f\|_{-(1+\vartheta)}.
\end{displaymath}
The estimate (\ref{eq7.23}) and generally for
$0\leq s\leq 1+\vartheta$ and $t=s-(1+\vartheta)$ 
the estimate
\begin{displaymath}
\|G_kPf\|_{\,s}\,\leq\,c\,\|f\|_{\,t}
\end{displaymath}
follow from these two estimates again by 
interpolation.
\qed
\end{proof}
Finally we combine the estimates from the last two 
lemmata to an estimate of the norm of the composed 
operators $G_\ell PV_k$ from the solution space 
$H^1$ back into itself. Introducing for 
abbreviation the prefactor
\begin{equation}    \label{eq7.24}
\alpha\;\,=\;\,\frac{\kappa^*(\lambda,\vartheta)
\kappa(\vartheta)(\vartheta h)^2\,
(2\,Z+N-1)\,N^{\,(1+\vartheta)/2}}{4\,\sqrt{2}}\;\,
\gamma^{\,1/2-\vartheta}
\end{equation}
and moreover the constant
\begin{equation}    \label{eq7.25}
q\;=\;\exp\left(-\,\frac{\vartheta h}{2}\,\right),
\end{equation}
this estimate for the norm of the composed 
operators reads
\begin{equation}    \label{eq7.26}
\|G_\ell PV_k\,u\|_1\,\leq\,\alpha\,q^{\,|k|+|\ell|}\|u\|_1.
\end{equation}
The decisive point is that for an appropriate choice 
of the parameter $\gamma$ fixing the splitting of 
the wavefunctions into the smooth and the singular 
part the constant $\alpha$ can be made arbitrarily 
small and pushed below every bound. The reason is
that the operators $V_k$ map $H^1$ to $H^{-\vartheta}$, 
but the $G_\ell$ conversely $H^{-\vartheta}$ to the 
space $H^{\,2-\vartheta}$ of higher regularity, 
which brings the cut-off operator $P$ into play. For 
ease of presentation, we still relabel the operators 
$G_\ell PV_k$ and denote them as
\begin{equation}    \label{eq7.27}
T_{n,\ell}, \quad 
n\,=\,0,1,2,\ldots,\;\;\;\ell\,=\,1,\ldots,\ell(n).
\end{equation}
The index $n$ is associated with the exponential 
decay of their norms, such that 
\begin{equation}    \label{eq7.28}
\|T_{n,\ell}\,u\|_1\,\leq\,\alpha\,q^{\,n}\|u\|_1
\end{equation}
for $u\in H^1$. The index $\ell$ counts the operator 
products for which this estimate holds; there are 
$\ell(n)=\max(1,4n)$ pairs of integers $k$ and 
$\ell$ for which $|k|+|\ell|=n$.

The aim of this work is to examine how well the 
solution $u$ of the equation
\begin{equation}    \label{eq7.29}
u\;+\,\TS u\;=\,f
\end{equation}
can be approximated by a linear combination of Gauss
functions in terms of the corresponding approximation
properties of the right hand side $f$. The two series 
\begin{equation}    \label{eq7.30}
\VS\;=\sum_{k=-\infty}^\infty\!V_k, \quad
\GS P\;=\sum_{\ell=-\infty}^\infty\!G_\ell P
\end{equation}
of operators from $H^1$ to $H^{-\vartheta}$ and 
from $H^{-\vartheta}$ back to $H^1$ converge 
absolutely because of the exponential decay of the 
norms of the operators and the completeness of the 
corresponding spaces of linear operators. Their 
product
\begin{equation}    \label{eq7.31}
\TS\,=\;\bigg(\sum_{\ell=-\infty}^\infty\!G_\ell P\bigg)\!
\bigg(\sum_{k=-\infty}^\infty\!V_k\bigg)
\end{equation}
can thus be written as Cauchy product, that is,
in terms of the operators (\ref{eq7.27}) as 
\begin{equation}    \label{eq7.32}
\TS\,=\,\sum_{n=0}^\infty\sum_{\ell=1}^{\ell(n)}T_{n,\ell}.
\end{equation}
As follows from the summation formula
\begin{displaymath}    
\sum_{n=0}^\infty \ell(n)\,q^{\,n}\,=\;
1\,+\,4\sum_{n=1}^\infty n\,q^{\,n}\,=\;
\left(\frac{1+q}{1-q}\,\right)^2
\end{displaymath}
and the estimate (\ref{eq7.28}) for the norms of the
operators $T_{n,\ell}$,
\begin{equation}    \label{eq7.33}
\sum_{n=0}^\infty\sum_{\ell=1}^{\ell(n)}\|T_{n,\ell}\|_1
\;\leq\;\alpha\left(\frac{1+q}{1-q}\,\right)^2.
\end{equation}
The operator $\TS$ from $H^1$ into itself is therefore 
a contraction for sufficiently small $\alpha$ and 
$\gamma$, respectively. The equation (\ref{eq7.29}) 
possesses then for given $f\in H^1$ a unique solution 
$u\in H^1$. Note, however, that the condition on 
$\gamma$ resulting from (\ref{eq7.33}) is more 
restrictive than the conditions discussed in 
Sect.~\ref{sec3} and Sect.~\ref{sec4}.


\section{The approximate solution of the substitute equation}
\label{sec8}

\setcounter{equation}{0}
\setcounter{lemma}{0}
\setcounter{theorem}{0}

Our aim is to study how well the solution of the equation
(\ref{eq7.29}) can be approximated by a linear combination 
of Gauss functions in terms of corresponding approximation 
properties of the right hand side $f$. The solution 
possesses the representation
\begin{equation}    \label{eq8.1}  
u\;=\;\sum_{\nu=0}^\infty (-1)^\nu \TS^\nu\!f.
\end{equation}
In the first step of our analysis we study the 
approximability of the terms $\TS^\nu\!f$.

Let $u\in H^1$ be a function that can be well approximated 
by low numbers of Gauss functions in the following sense. 
Assume that there is an infinite sequence $u_1,u_2,\ldots$ 
of Gauss functions such that, for every $\varepsilon>0$,
\begin{equation}    \label{eq8.2}
\Big\|\,u\,-\sum_{j=1}^{n(\varepsilon)}u_j\,\Big\|_1
\leq\; \varepsilon,
\quad
n(\varepsilon)\,\leq\,\Big(\frac{\kappa}{\varepsilon}\Big)^{1/r},
\end{equation}
where $r>0$ is a given approximation order and $\kappa$ a constant 
that depends on $u$ but is independent of $\varepsilon$. If the 
integer $n(\varepsilon)\geq 0$ takes the value zero, the 
approximating sum is empty and the norm of $u$ itself already 
less than or equal to~$\varepsilon$. We can assume that 
$n(\varepsilon)$ increases when $\varepsilon$ decreases; if 
necessary, one replaces $n(\varepsilon)$ simply by the minimum 
of all $n(\varepsilon')$ for $\varepsilon'\leq\varepsilon$. Our 
first objective is to show that, with sufficiently small values 
of the constant $\alpha$ from (\ref{eq7.24}), one can approximate 
the function
\begin{equation}    \label{eq8.3}
\TS u\;=\,\sum_{k=0}^\infty\sum_{\ell=1}^{\ell(k)}T_{k,\ell}\,u
\end{equation}
basically by half the number of Gauss functions with 
double accuracy.

These approximations are constructed as follows. First, 
we split the decay rate (\ref{eq7.25}) in dependence 
of the order $r$ into the product $q=q_1q_2$ of two constants
less than one chosen such that our later estimates become best 
possible. They are given~by   
\begin{equation}    \label{eq8.4}
q_1\,=\;\exp\left(-\,\frac{1}{r+1}\,\frac{\vartheta h}{2}\,\right),
\quad
q_2\,=\;\exp\left(-\,\frac{r}{r+1}\,\frac{\vartheta h}{2}\,\right).
\end{equation}
The approximations of $\TS u$ are then the sums
\begin{equation}    \label{eq8.5}
\sum_{k=0}^\infty\sum_{\ell=1}^{\ell(k)}\sum_{j=1}^{n_k}
T_{k,\ell}\,u_j,
\quad
n_k\,=\,n(\delta^{-1}q_2^{-k}\varepsilon)
\end{equation}
of Gauss functions, where $\delta>0$ is here a new parameter 
that still needs to be fixed. These triple sums are finite 
because $n_k=0$ and the inner sums are thus empty as soon as 
the error bound $\delta^{-1}q_2^{-k}\varepsilon$ becomes 
greater than $\kappa$. The approximation error and the number 
of the remaining nonzero terms are estimated in next lemma.
\begin{lemma}       \label{lm8.1}
Let $\varepsilon>0$ be arbitrary and choose the $n_k$ 
as in \rmref{eq8.5}. Then
\begin{equation}    \label{eq8.6}
\bigg\|\;
\sum_{k=0}^\infty\sum_{\ell=1}^{\ell(k)}T_{k,\ell}\,u
\;-\,
\sum_{k=0}^\infty\sum_{\ell=1}^{\ell(k)}\sum_{j=1}^{n_k}
T_{k,\ell}\,u_j
\;\bigg\|_1
\leq\;
\frac{\alpha}{\delta}\,
\left(\frac{1+q_1}{1-q_1}\right)^2\,\varepsilon,
\end{equation}
where $\alpha$ is the constant from \rmref{eq7.24}. 
Moreover, 
\begin{equation}    \label{eq8.7}
\sum_{k=0}^\infty\sum_{\ell=1}^{\ell(k)}n_k
\;\leq\;
\delta^{1/r}\left(\frac{1+q_1}{1-q_1}\right)^2\,
\Big(\frac{\kappa}{\varepsilon}\Big)^{1/r}
\end{equation}
holds for the number of terms $T_{k,\ell}\,u_j$ of 
which the approximation \rmref{eq8.5} is composed.
\end{lemma}
\begin{proof}
The left hand side of (\ref{eq8.6}) can obviously 
be estimated by the double sum
\begin{displaymath}
\sum_{k=0}^\infty\sum_{\ell=1}^{\ell(k)}\,
\Big\|\,
T_{k,\ell}\Big(u\,-\sum_{j=1}^{n_k}u_j\,\Big)
\Big\|_1.
\end{displaymath}
The error estimate (\ref{eq8.6}) thus follows from
(\ref{eq7.28}), the assumption (\ref{eq8.2}), that is,
\begin{displaymath}     
\Big\|\,u\,-\sum_{j=1}^{n_k}u_j\,\Big\|_1
\leq\; \delta^{-1}q_2^{-k}\varepsilon
\end{displaymath}
in the present case, from $\|T_{k,\ell}\|_1\leq\alpha\,q^k$,
$qq_2^{-1}=q_1$, $\ell(k)=\max(1,4k)$, and 
\begin{displaymath}
\sum_{k=0}^\infty \ell(k)\,q_1^{\,k}\,=\,
\left(\frac{1+q_1}{1-q_1}\right)^2.
\end{displaymath}
Because $q_2^{1/r}=q_1$, the estimate (\ref{eq8.7}) 
for the number
\begin{displaymath}
\sum_{k=0}^\infty\sum_{\ell=1}^{\ell(k)} n_k
\;\leq\;
\sum_{k=0}^\infty \ell(k)\,
\kappa^{1/r}(\delta^{-1}q_2^{-k}\varepsilon)^{-1/r}
\end{displaymath}
of terms in the approximation (\ref{eq8.5}) results 
with the same summation formula.
\qed
\end{proof}
The operators $T_{k,\ell}$ map a Gauss function to 
a sum of $M/2$ Gauss functions, where
\begin{equation}    \label{eq8.8}
M\,=\;4\,\left(KN\,+\,\frac{(N-1)N}{2}\,\right)
\end{equation}
is the quadruple of the number of the interaction terms 
between the electrons and the nuclei and the electrons 
among each other and $K$ is the number of the nuclei.
The estimates from Lemma~\ref{lm8.1} suggest therefore 
to choose
\begin{equation}    \label{eq8.9}
\delta\,=\,\frac{1}{M^r}
\left(\frac{1-q_1}{1+q_1}\right)^{2r},
\quad
\alpha\,\leq\,\frac{1}{2\,M^r}
\left(\frac{1-q_1}{1+q_1}\right)^{2r+2}.
\end{equation}
The error estimate (\ref{eq8.6}) then reduces to
\begin{equation}    \label{eq8.10}
\bigg\|\;
\sum_{k=0}^\infty\sum_{\ell=1}^{\ell(k)}T_{k,\ell}\,u
\;-\,
\sum_{k=0}^\infty\sum_{\ell=1}^{\ell(k)}\sum_{j=1}^{n_k}
T_{k,\ell}\,u_j
\;\bigg\|_1
\leq\;\,
\frac{\varepsilon}{2},
\end{equation}
and the number of Gauss functions in the approximation 
(\ref{eq8.5}) is bounded by
\begin{equation}    \label{eq8.11}
\frac{M}{2}\,\sum_{k=0}^\infty\sum_{\ell=1}^{\ell(k)} n_k
\;\leq\;
\frac{1}{2}\,\Big(\frac{\kappa}{\varepsilon}\Big)^{1/r}.
\end{equation}
Thus our goal is reached and we have shown that $\TS u$ 
can be approximated by half number of Gauss functions
with double accuracy, provided the width of the 
smoothing kernel (\ref{eq3.5}) is sufficiently small, 
such that the condition on $\alpha$ from (\ref{eq8.9}) 
holds. This condition limits the size of $\alpha$ the 
more the larger the approximation order $r$ becomes. 
In the limit $r=0$, it turns into
\begin{equation}    \label{eq8.12}
\alpha\,\leq\,\frac{1}{2}\,\left(\frac{1-q}{1+q}\,\right)^2.
\end{equation} 
It implies by (\ref{eq7.33}) therefore the estimate 
$\|\TS\|_1\leq1/2$ for the operator $\TS$ from the 
space $H^1$ into itself and thus ensures the 
convergence of the series (\ref{eq8.1}) in $H^1$. 

We have assumed that the number $n(\varepsilon)$ of terms 
in (\ref{eq8.2}) needed to reach an error of norm 
$\leq\varepsilon$ in the approximation of the given function 
$u$ increases when $\varepsilon$ decreases. A decreasing 
$\varepsilon$ thus means that further terms are added to the 
sum (\ref{eq8.5}). Sorting and numbering the single terms 
correspondingly, one gets therefore a new sequence of Gauss 
functions $w_j$ such that, for all 
$\varepsilon>0$,
\begin{equation}    \label{eq8.13}
\Big\|\,\TS u\,-\sum_{j=1}^{n_1(\varepsilon)}w_j\,\Big\|_1
\leq\; \varepsilon,
\quad
n_1(\varepsilon)\,\leq\,
\frac{1}{2}\,\Big(\frac{\kappa}{2\,\varepsilon}\Big)^{1/r}.
\end{equation}
The old situation is thus restored, but with a new function
$n_1(\varepsilon)$ counting the number of Gauss functions
needed to obtain a given accuracy. As the constant $\kappa$
from (\ref{eq8.2}) does not enter into the bound for 
$\alpha$, this process can be iterated.
\begin{lemma}       \label{lm8.2}
Starting from a function $u\in H^1$ as in \rmref{eq8.2}, 
for every $\nu=1,2,\ldots$ and every $\varepsilon>0$, 
the function $\TS^\nu u$ can be approximated by a linear 
combination of
\begin{equation}    \label{eq8.14}
n_\nu(\varepsilon)\,\leq\,
\frac{1}{2^\nu}\,\Big(\frac{\kappa}{2^\nu\varepsilon}\Big)^{1/r}
\end{equation}
Gauss functions up to an $H^1$-error $\varepsilon$, 
provided $\alpha$ is chosen as in \rmref{eq8.9}.
\end{lemma}

Such a small norm of the operator $\TS$ as enforced 
by the condition from (\ref{eq8.9}) to the constant 
$\alpha$ means that only very few terms 
$-\TS\!f,\TS^2\!f,\ldots$ need to be added to the 
right hand side of the equation $u+\TS u=f$ to 
approximate its solution
\begin{equation}    \label{eq8.15}
u\;=\,f\,-\,\TS\!f\,+\,\TS^2\!f\,-\;\ldots
\end{equation}
with high accuracy. The first term $-\TS\!f$ explicitly 
depends on the distances of the electrons and covers 
two-particle interactions, the second then also 
three-particle interactions, and so on. Lemma~\ref{lm8.2} 
means that less and less Gauss functions are necessary 
to approximate these terms sufficiently well and leads 
to our final and concluding
\begin{theorem}     \label{thm8.1}
Let $f$ be a function in $H^1$ and assume that there 
exists an infinite sequence of Gauss functions 
$g_1,g_2,\ldots$ such that, for every $\varepsilon>0$,
\begin{equation}    \label{eq8.16}
\Big\|\,f\,-\sum_{j=1}^{n(\varepsilon)}g_j\,\Big\|_1
\leq\; \varepsilon,
\quad
n(\varepsilon)\,\leq\,\Big(\frac{\kappa}{\varepsilon}\Big)^{1/r},
\end{equation}
where $r$ is a given approximation order and $\kappa$ a 
constant that depends on $f$ but is independent of 
$\varepsilon$. The solution $u\in H^1$ of the equation
\begin{equation}    \label{eq8.17}
u\;+\,\TS u\,=\,f
\end{equation}
can then, for every $\varepsilon>0$, be approximated 
by a linear combination of 
\begin{equation}    \label{eq8.18}
n\,\leq\,2\;\Big(\frac{2\kappa}{\varepsilon}\Big)^{1/r}
\end{equation}
Gauss functions up to an $H^1$-error $\varepsilon$, 
provided the width of the smoothing kernel \rmref{eq3.5} 
is sufficiently small in dependence of the approximation 
order $r$, that is, the constant $\alpha$ from 
\rmref{eq7.24} satisfies the condition from \rmref{eq8.9}.
\end{theorem}
\begin{proof}
The proof is based on the representation
\begin{displaymath}
(I+\TS)^{-1}f\;=\,\sum_{\nu=0}^\infty (-1)^\nu \TS^\nu\!f
\end{displaymath}
and the approximation of the parts $\TS^\nu\!f$ up to an 
error $\varepsilon_\nu\,=\;2^{-(\nu+1)}\varepsilon$. By 
Lemma~\ref{lm8.2}, the needed number of Gauss functions 
sums then up to a value not larger than 
\begin{displaymath}
\sum_{\nu=0}^\infty\;
\frac{1}{2^\nu}\,\Big(\frac{\kappa}{2^\nu\varepsilon_\nu}\Big)^{1/r}
\!=\;\,2\;\Big(\frac{2\kappa}{\varepsilon}\Big)^{1/r}
\end{displaymath}
and the errors $\varepsilon_\nu$ to the target accuracy 
$\varepsilon$, provided the width of the smoothing kernel 
(\ref{eq3.5}) is so small that the constant $\alpha$ 
satisfies the condition from (\ref{eq8.9}).
\qed
\end{proof}
We remark that the theorem can be generalized from $H^1$ 
to every space $H^{\,s}$ of order $0\leq s\leq 1+\vartheta$,  
under the same condition to the constant $\alpha$. 
The proof starts from a generalization of Lemma~\ref{lm7.7} 
and Lemma~\ref{lm7.8}, considering the $V_k$ as operators 
from $H^{\,s}$ to $H^{\,t}$, $t=s-(1+\vartheta)$, and the 
$G_kP$ correspondingly as operators from $H^{\,t}$ back 
to the space $H^{\,s}$ of interest, and proceeds then 
as before. 

The given approximations of the solution (\ref{eq8.1}) 
can also be seen from a different perspective, 
decomposing them not into single Gauss functions but 
into a much smaller number of generic building blocks. 
Let us call these building blocks elementary functions 
and let us assume that the right hand side $f$ is 
given as a series of elementary functions of level zero, 
say as a series of Gauss functions as considered so far, 
or, for example, a series of Slater determinants of 
three-dimensional Gaussian orbitals. If $g$ is an 
elementary function of level~$\nu$, the application 
of the operators $T_{k,\ell}$ to $g$ yields elementary 
functions of level $\nu+1$. That is, the $\TS^\nu\!f$ 
are composed of elementary functions of level $\nu$, 
and their approximations of finite subsets of these 
functions. The symmetry properties of the elementary 
functions with respect to the exchange of the electron 
positions are inherited from one level to the next. 
Starting from corresponding approximation properties 
of the series representing the right hand side, the 
total number of elementary functions needed to 
approximate the solution up to a given accuracy 
can be estimated in exactly the same way as this has 
been done here for Gauss functions, where the number 
of elementary functions halves again from one level 
to the next. The condition from (\ref{eq8.9}) needs 
only to be replaced by the weaker and less restrictive 
condition
\begin{equation}    \label{eq8.19}
\alpha\,\leq\,
\frac{1}{2^{r+1}}\left(\frac{1-q_1}{1+q_1}\right)^{2r+2}
\end{equation}
into which the system parameters and the eigenvalue
enter only indirectly via the definition 
(\ref{eq7.24}) of the constant $\alpha$. With this 
new condition to $\alpha$, Theorem~\ref{thm8.1} 
literally transfers to the present situation. The 
price to be paid is the more complex, problem-dependent 
structure of the single building blocks. The compact 
representation and compression of these building blocks 
needs further investigation.

In view of our application to the electronic wavefunction $u$ 
under consideration, the solution of the equation (\ref{eq4.7}), 
the condition to the width of the smoothing kernel~$K$ means 
that a sufficiently large part of $u$ needs to be shifted to 
the right hand side \mbox{$f=Qu$} of the equation. The 
$H^1$-distance between the wavefunction $u$ and its smoothed 
variant $Qu=K*u$ is by (\ref{eq4.10}) of order $\gamma^{\,1/2}$. 
The $H^1$-distance of the wavefunction $u$ and the solution 
$\uS$ of the perturbed equation (\ref{eq4.8}):
\begin{equation}    \label{eq8.20}
\uS\,+\TS\uS\,=\,f, \quad f\,=\,Qu,
\end{equation}
behaves by (\ref{eq4.9}), by (\ref{eq6.6}) and (\ref{eq6.11}),
and by (\ref{eq5.10}) and (\ref{eq5.11}) in comparison like
\begin{equation}    \label{eq8.21}
\lesssim\;\gamma^{\,1/2}\,h^{-1/2}\e^{-\pi^2/h}.
\end{equation}
For small $h$, the gap between the basic accuracy of order 
$\gamma^{\,1/2}$ and the attainable accuracy thus rapidly 
widens to many orders of magnitude. The conclusion from 
Theorem~\ref{thm8.1} is therefore that the approximation 
of the quasi-exact solution $\uS$ of the Schr\"odinger 
equation does not require a substantially larger number 
of terms than that of the smoothed variant $Qu$ of the 
true wavefunction. The question remains when this rather
astonishing effect actually sets in and how far the also 
with the best possible choice of the parameter 
$\vartheta<1/2$ still very stringent condition on 
the width of the smoothing kernel can be relaxed.


\section{Epilogue. First steps toward a numerical procedure}
\label{sec9}

\setcounter{equation}{0}
\setcounter{lemma}{0}
\setcounter{theorem}{0}

Many difficulties still have to be overcome on the 
way to a numerical method that fully exploits the 
approximation properties of the given class of Gauss 
and Gauss-Hermite functions and that enables to compute 
such approximations efficiently. This begins with the 
mentioned problem how to incorporate the symmetry 
properties enforced by the Pauli principle and how 
to store such antisymmetrized functions in compact 
form; there is no such thing as Slater determinants. 
However, there are some basic components that will 
presumably be part of such methods. One is approximate 
inverse iteration, a procedure that evolved in recent 
years into a very popular method for the solution of 
the large matrix eigenvalue problems that arise from 
the discretization of linear selfadjoint elliptic 
partial differential equations. The analysis of such 
methods essentially started with the work of D'yakonov 
and Orekhov~\cite{Dyakonov-Orekhov}. In a series of 
groundbreaking papers, Knyazev and Neymeyr analyzed 
these methods in great detail; we refer to 
\cite{Knyazev-Neymeyr} and the literature cited 
therein. 

Approximate inverse iteration can be directly applied 
to operator equations in infinite dimensional spaces
\cite{Rohwedder-Schneider-Zeiser} and can best be 
understood in terms of the weak formulation of the
eigenvalue problems. Let $\Hs$ be a Hilbert space that 
is equipped with the inner product $a(u,v)$ inducing 
the energy norm $\|u\|$, under which it is complete, 
and a further inner product $(u,v)$. Let the infimum 
of the Rayleigh quotient 
\begin{equation}    \label{eq9.1}
\lambda(u)=\frac{a(u,u)}{(u,u)}, \quad 
\text{$u\neq 0$ in $\Hs$},
\end{equation}
be an isolated eigenvalue $\lambda_1>0$ of finite 
multiplicity and let $\E_1$ be the assigned 
eigenspace, the finite dimensional space of all 
$u\in\Hs$ for which 
\begin{equation}    \label{eq9.2}
a(u,v)=\lambda_1(u,v), \quad v\in\Hs,
\end{equation}
or equivalently $\lambda(u)=\lambda_1$ holds. 
The aim is the calculation of this eigenvalue, 
that is, the ground state energy of the system 
under consideration.

Let $\lambda_2>\lambda_1$ be the infimum of the
Rayleigh quotient on the with respect to both 
inner products orthogonal complement of the 
eigenspace $\E_1$ for the eigenvalue $\lambda_1$.
In cases like ours, $\lambda_2$ is also an 
isolated eigenvalue, but this is not truly needed.
Given an element $u\in\Hs$~with norm $\|u\|_0=1$ 
and Rayleigh quotient $\lambda(u)<\lambda_2$, 
in inverse iteration in its original, exact 
version at first the solution $w\in\Hs$ of the 
equation
\begin{equation}    \label{eq9.3}
a(w,v)=\,a(u,v)-\lambda(u)(u,v), \quad v\in\Hs,
\end{equation}
is determined, which exists by the Riesz 
representation or the Lax-Milgram theorem and 
is unique. The current $u$ is then replaced 
by $u-w$. Since $a(u,w)=0$, the new element 
$u-w$ is different from zero so that 
$\lambda(u-w)$ is well defined and the process 
can be repeated with the normed version of $u-w$. 
The so iteratively generated sequence of 
Rayleigh quotients decreases then monotonously to 
the eigenvalue $\lambda_1$ and the iterates~$u$ 
converge to an eigenvector or eigenfunction for 
this eigenvalue.

In the approximate version of the method, the 
solution $w$ of equation (\ref{eq9.3}) is replaced 
by an approximation $\wS\in\Hs$ for which an error 
estimate
\begin{equation}    \label{eq9.4}
\|\wS-w\|\leq\delta\|w\|
\end{equation}
holds, where $\delta<1$ is a fixed constant that 
controls the accuracy. Then $u-\wS\neq 0$, so 
that the process can proceed with the normed 
version $u'$ of $u-\wS$ as new iterate. No 
assumption on the origin of $\wS$ is needed. 
It can, for example, be the element of best 
approximation of $w$ in a finite dimensional 
subspace of $\Hs$, an iteratively calculated 
approximation of this element, or anything 
else wherever it comes from. 

A simple, albeit surely not optimal analysis 
of this variant along the lines given in the 
original paper of  D'yakonov and Orekhov can 
be found in \cite{Ys_2016}. Main result is that 
under the given assumptions, and if in particular 
already $\lambda(u)<\lambda_2$, the estimate
\begin{equation}    \label{eq9.5}
\lambda(u')-\lambda_1\,\leq\,q(\lambda(u))(\lambda(u)-\lambda_1)
\end{equation}
holds, where $q(\lambda)$ is the on the interval 
$\lambda_1\leq\lambda\leq\lambda_2$ strictly
increasing function
\begin{equation}    \label{eq9.6}
q(\lambda)=\,1-\,
\frac{(1-\delta^2)\,\lambda\,(\lambda_2-\lambda)^2}
{\lambda_2^2\,\lambda+(1-\delta^2)(\lambda_2-\lambda)^2(\lambda-\lambda_1)}.
\end{equation}
If one starts therefore with a normed $u=u_0$ in $\Hs$
with Rayleigh quotient $\lambda(u_0)<\lambda_2$ and
generates as described a sequence of normed $u_k$, the
Rayleigh quotients $\lambda(u_k)$ decrease strictly
to the minimum eigenvalue $\lambda_1$ or become 
stationary there. Moreover, one can show that the 
iterates $u_k$ converge to an eigenvector for the 
eigenvalue $\lambda_1$. 

To apply this form of inverse iteration to our 
eigenvalue problem (\ref{eq2.6}), we have first
to shift the Hamiltonian and to replace the
bilinear form (\ref{eq2.5}) by a bilinear form
\begin{equation}    \label{eq9.7}
a(u,v)=\int\big\{\nabla u\cdot\nabla v+Vuv+\mu uv\big\}\dx
\end{equation}
on $\Hs=H^1$, with $\mu>0$ a still to be determined 
constant value. The original eigenvalues $\lambda$ 
turn then into $\lambda+\mu$; the eigenfunctions 
themselves remain untouched. The term $a(u,u)$ can, 
according to Lemma~\ref{lm2.1}, be estimated from 
above and below by
\begin{equation}    \label{eq9.8}
\|\nabla u\|_0^2\,\pm\,\theta\|\nabla u\|_0\|u\|_0+\mu\|u\|_0^2,
\end{equation}
where $\|u\|_0$ is the $L_2$-norm and $\theta$
the norm (\ref{eq3.11}) of the potential $V$, 
understood as operator from $H^1$ to $L_2$. 
Making the expressions (\ref{eq9.8}) extremal 
under the constraint $\|\nabla u\|_0^2+\mu\|u\|_0^2=1$, 
one can therefore estimate $a(u,u)$ from below and 
above by
\begin{equation}    \label{eq9.9}
\left(1-\frac{\theta}{2\sqrt{\mu}}\right) b(u,u)
\,\leq\, a(u,u) \,\leq\,
\left(1+\frac{\theta}{2\sqrt{\mu}}\right) b(u,u)
\end{equation}
in terms of the inner product
\begin{equation}    \label{eq9.10}
b(u,v)=\int\big\{\nabla u\cdot\nabla v+\mu uv\big\}\dx
\end{equation}
on $H^1$, now without the potential part. The shifted 
bilinear form (\ref{eq9.7}) fits therefore for values 
$\mu>\theta^2/4$ into the described framework of 
approximate inverse iteration.

An obvious choice for the approximation of the 
solution $w\in H^1$ of equation (\ref{eq9.3}) 
is in the given context the solution 
$\wS\in H^1$ of the equation
\begin{equation}    \label{eq9.11}
b(\wS,v)=\,a(u,v)-\lambda(u)(u,v), \quad v\in H^1.
\end{equation}
By the definition of $w$ and $\wS$ and again 
by Lemma~\ref{lm2.1}, then the estimate
\begin{equation}    \label{eq9.12}
b(\wS-w,\wS-w)=(Vw,\wS-w)\leq\theta\|\nabla w\|_0\|\wS-w\|_0
\end{equation}
holds. With the help of (\ref{eq9.9}), finally 
the energy norm estimate
\begin{equation}    \label{eq9.13}
\|\wS-w\|\leq \sqrt{c(\eta)}\,\eta\,\|w\|
\end{equation}
in terms of $\eta=\theta/\sqrt{\mu}$ and the function 
$c(\eta)=(2+\eta)/(2-\eta)$ follows. That is, for  
sufficiently large constants $\mu$ the basic condition 
(\ref{eq9.4}) is fulfilled.

Now we have reached the point at which the Gauss 
functions come into play. With $u$ given, the 
new approximation reads, before normalization, 
in operator form
\begin{equation}    \label{eq9.14}
u-\wS\,=\,
u\,-\,(-\Delta+\mu)^{-1}(-\Delta u+\mu u+Vu-\lambda(u)u\,).
\end{equation}
The idea is to approximate the operators
$(-\Delta+\mu)^{-1}$ and $V$ in the discussed 
manner or similarly by means of Gauss functions. 
Gauss functions are then again mapped to series 
of Gauss functions, which have to be truncated 
appropriately. The residual 
\begin{equation}    \label{eq9.15}
-\Delta u+\mu u+Vu-\lambda(u)u
\end{equation}
actually does not depend on the choice of the 
constant $\mu$, which enters into the process 
only via the inverse of the shifted Laplace 
operator. It needs to be approximated with high 
accuracy to fulfill the condition (\ref{eq9.4}) 
on the accuracy of the approximation~$\wS$. 
The requirements for the approximation of the 
inverse of the shifted Laplace operator are 
in comparison modest. The price to be paid is 
that the degree of the polynomial part of the 
Gaussians increases due the Laplace part in the 
residual from one step to the next. Starting 
from the representation
\begin{equation}    \label{eq9.16}
u-\wS\,=\,(-\Delta+\mu)^{-1}(\lambda(u)u-Vu)
\end{equation}
of the new iterate before normalization this can 
be avoided, but conversely then also the inverse 
of the shifted Laplacian has to be approximated 
with high accuracy. The big unsolved question 
with both variants is how to truncate the 
intermediate series of Gauss functions to keep 
the number of terms on a computationally still 
feasible level, without sacrificing the accuracy 
or lowering it too much.


\section*{Appendix. Remarks on interpolation}

In Sect.~\ref{sec7}, more precisely in the proof of 
Lemma~\ref{lm7.7} and \ref{lm7.8}, we have used some 
results from interpolation theory. The interpolation
between Banach and Hilbert spaces is a large and 
well established field of functional analysis and 
the theory of function spaces. A standard reference 
is \cite{Bergh-Loefstroem}. A more condensed 
presentation coming closer to our needs can be found 
in \cite{McLean}. We are in the lucky situation that 
the proof of the mentioned two lemmata requires only 
the interpolation within the space of rapidly 
decreasing functions, a fact that simplifies the 
argumentation a lot and enables us to derive the 
necessary results in a few lines.

The key is the representation of the norms on the 
space $\S$ of the real-valued, rapidly decreasing 
functions in terms of the $K$-functionals 
\begin{equation*}  
K(t,u,\vartheta_1,\vartheta_2)\,=\;\inf_{v\in\S}
\big\{\,\|u-v\|_{\vartheta_1}^2+\;t^2\|v\|_{\vartheta_2}^2\big\}^{1/2},
\quad \vartheta_1<\vartheta_2,
\end{equation*}
which can be considered as more refined 
smoothness measures.
\begin{unnumberedlemma}
Let $\vartheta_1<\vartheta_2$, $0<s<1$, and
$\vartheta=\vartheta_1+s\,(\vartheta_2-\vartheta_1)$.
For all $u\in\S$ then
\begin{equation*}    
\int_0^\infty [\;t^{-s}K(t,u,\vartheta_1,\vartheta_2)\,]^2\,\frac{\dt}{t}
\;=\,\int_0^\infty\frac{\,t^{1-2s}}{1+\,t^2}\,\dt\;\|u\|_\vartheta^{\,2}.
\end{equation*}
\end{unnumberedlemma}
\begin{proof}
The expression whose infimum is sought reads in 
Fourier representation
\begin{displaymath}
\int\Big\{\big(1+|\vomega|^2\big){}^{\vartheta_1}
|\fourier{u}(\vomega)-\fourier{v}(\vomega)|^2\,+\;
t^2\big(1+|\vomega|^2\big){}^{\vartheta_2}
|\fourier{v}(\vomega)|^2\,\Big\}\domega.
\end{displaymath}
The integrand is, with given $\fourier{u}(\vomega)$, 
pointwise minimized by the value
\begin{displaymath}
\fourier{v}(\vomega)\;=\;
\frac{\fourier{u}(\vomega)}
{1+\,t^2\big(1+|\vomega|^2\big){}^{\vartheta_2-\vartheta_1}}.
\end{displaymath}
This expression defines another real-valued, rapidly 
decreasing function $v$ at which the infimum is 
attained. Inserting this function above, we get a 
closed representation of the $K$-functional 
$K(t,u,\vartheta_1,\vartheta_2)$ of $u$ in terms 
of the Fourier transform of $u$:
\begin{displaymath}
K(t,u,\vartheta_1,\vartheta_2)^2\,=\,
\int\frac{t^2\big(1+|\vomega|^2\big){}^{\vartheta_2-\vartheta_1}}
{1+\,t^2\big(1+|\vomega|^2\big){}^{\vartheta_2-\vartheta_1}}\;
\big(1+|\vomega|^2\big){}^{\vartheta_1}
|\fourier{u}(\vomega)|^2\domega.
\end{displaymath}
The proposition follows from this representation 
with Fubini's theorem.
\qed
\end{proof}
Let us assume now that we have a linear operator 
$T:\S\to\S$ and that for all $u\in\S$
\begin{equation*}
\|Tu\|_{\vartheta_1'}\leq c\,\|u\|_{\vartheta_1},
\quad
\|Tu\|_{\vartheta_2'}\leq c\,\|u\|_{\vartheta_2},
\end{equation*}
where $\vartheta_1<\vartheta_2$ and $\vartheta_1'<\vartheta_2'$
are two arbitrarily given pairs of real numbers.
\begin{unnumberedlemma}
Let $0\leq s\leq 1$, 
$\vartheta=\vartheta_1+s\,(\vartheta_2-\vartheta_1)$, and
$\vartheta'=\vartheta_1'+s\,(\vartheta_2'-\vartheta_1')$.
For all real-valued, rapidly decreasing functions $u$ 
then, with the same constant, 
\begin{equation*}
\|Tu\|_{\vartheta'}\leq c\,\|u\|_{\vartheta}.
\end{equation*}
\end{unnumberedlemma}
\begin{proof}
For $s=0$ and $s=1$, the proposition already holds
by assumption. For the remaining values in between  
it follows immediately from the estimate 
\begin{displaymath}
K(t,Tu,\vartheta_1',\vartheta_2')\,\leq\,
c\,K(t,u,\vartheta_1,\vartheta_2)
\end{displaymath}
for the $K$-functionals of $u$ and $Tu$ assigned to the 
given norms and the representation above of the two 
intermediate norms in terms of these $K$-functionals.
\qed
\end{proof}


\bibliographystyle{spmpsci}
\bibliography{paper}


\end{document}